\pdfoutput=1 % FOR ARXIV
\documentclass[a4paper,egregdoesnotlikesansseriftitles,final]{scrartcl}
\usepackage{amsthm,amssymb,amsmath,enumerate,graphicx}
\usepackage{tikz}
\usepackage[utf8]{inputenc}

\usepackage{thm-restate}

\newtheorem{THM}{Theorem}[section]
\newtheorem{LEM}[THM]{Lemma}

\newtheorem{COR}[THM]{Corollary}
\newtheorem{PROP}[THM]{Proposition}

\theoremstyle{definition}
\newtheorem{EX}[THM]{Example}

\def\shift(#1)(#2){\!\!\downarrow\!{}^{#1}_{\raise .1ex\vbox to 0pt{\vss\hbox{$\scriptstyle #2$}}}\,}
\def\ucl(#1){\lfloor #1 \rfloor}% up-closure
\def\dcl(#1){\lceil #1 \rceil}% down-closure
\def\specrel#1#2{\mathrel{\mathop{\kern0pt #1}\limits_{#2}}}

\newcommand\B{\mathcal B}

\def\S{\mathcal S}

\def\curlyle{\preccurlyeq}

\def\lowfwd #1#2#3{{\mathop{\kern0pt #1}\limits^{\kern#2pt\raise.#3ex
\vbox to 0pt{\hbox{$\scriptscriptstyle\rightarrow$}\vss}}}}
\def\lowbkwd #1#2#3{{\mathop{\kern0pt #1}\limits^{\kern#2pt\raise.#3ex
\vbox to 0pt{\hbox{$\scriptscriptstyle\leftarrow$}\vss}}}}
\def\fwd #1#2{{\lowfwd{#1}{#2}{15}}}
\def\ve{\kern-1.5pt\lowfwd e{1.5}2\kern-1pt}
\def\vedash{{\mathop{\kern0pt e\lower.5pt\hbox{${}% logically \v(e')
     \scriptstyle'$}}\limits^{\kern0pt\raise.02ex
     \vbox to 0pt{\hbox{$\scriptscriptstyle\rightarrow$}\vss}}}}
\def\ev{\kern-1pt\lowbkwd e{0.5}2\kern-1pt}
\def\vf{\kern-2pt\lowfwd f{2.5}2\kern-1pt}
\def\vfdash{{\mathop{\kern0pt f\raise 1pt\hbox{${}% logically \v(f')
     \scriptstyle'$}}\limits^{\kern2pt\raise.02ex
     \vbox to 0pt{\hbox{$\scriptscriptstyle\rightarrow$}\vss}}}}

\def\vp{\lowfwd p{1.5}2}

\def\vr{\lowfwd r{1.5}2}
\def\rv{\lowbkwd r02}

\def\vrdash{{\mathop{\kern0pt r\lower.5pt\hbox{${}% logically \v(e')
     \scriptstyle'$}}\limits^{\kern0pt\raise.02ex
     \vbox to 0pt{\hbox{$\scriptscriptstyle\rightarrow$}\vss}}}}
\def\rvdash{{\mathop{\kern0pt r\lower.5pt\hbox{${}% logically \v(e')
     \scriptstyle'$}}\limits^{\kern0pt\raise.02ex
     \vbox to 0pt{\hbox{$\scriptscriptstyle\leftarrow$}\vss}}}}
\def\vrdashp{{\mathop{\kern0pt r_p\kern-4pt\lower.5pt\hbox{${}% logically \v(e')
     \scriptstyle'$}}\limits^{\kern0pt\raise.02ex
     \vbox to 0pt{\hbox{$\scriptscriptstyle\rightarrow$}\vss}}}\,}
\def\rvdashp{{\mathop{\kern0pt r_p\kern-4pt\lower.5pt\hbox{${}% logically \v(e')
     \scriptstyle'$}}\limits^{\kern0pt\raise.02ex
     \vbox to 0pt{\hbox{$\scriptscriptstyle\leftarrow$}\vss}}}\,}
\def\vrddash{{\mathop{\kern0pt r\lower.5pt\hbox{${}% logically \v(e')
     \scriptstyle''$}}\limits^{\kern0pt\raise.02ex
     \vbox to 0pt{\hbox{$\scriptscriptstyle\rightarrow$}\vss}}}}
\def\vrone{\lowfwd {r_1}12}

\def\vrtwo{\lowfwd {r_2}12}

\def\vs{\lowfwd s{1.5}1}
\def\sv{{{\lowbkwd s{1.5}1}\hskip-1pt}}
\def\vsdash{{\mathop{\kern0pt s\lower.5pt\hbox{${}% logically \v(e')
     \scriptstyle'$}}\limits^{\kern0pt\raise.02ex
     \vbox to 0pt{\hbox{$\scriptscriptstyle\rightarrow$}\vss}}}}
\def\svdash{{\mathop{\kern0pt s\lower.5pt\hbox{${}% logically \v(e')
     \scriptstyle'$}}\limits^{\kern0pt\raise.02ex
     \vbox to 0pt{\hbox{$\scriptscriptstyle\leftarrow$}\vss}}}}
\def\vsddash{{\mathop{\kern0pt s\lower.5pt\hbox{${}% logically \v(e')
     \scriptstyle''$}}\limits^{\kern0pt\raise.02ex
     \vbox to 0pt{\hbox{$\scriptscriptstyle\rightarrow$}\vss}}}}
\def\svddash{{\mathop{\kern0pt s\lower.5pt\hbox{${}% logically \v(e')
     \scriptstyle''$}}\limits^{\kern0pt\raise.02ex
     \vbox to 0pt{\hbox{$\scriptscriptstyle\leftarrow$}\vss}}}}
\def\vsdashp{{\mathop{\kern0pt s_p\kern-4pt\lower.5pt\hbox{${}% logically \v(e')
     \scriptstyle'$}}\limits^{\kern0pt\raise.02ex
     \vbox to 0pt{\hbox{$\scriptscriptstyle\rightarrow$}\vss}}}\,}
\def\svdashp{{\mathop{\kern0pt s_p\kern-4pt\lower.5pt\hbox{${}% logically \v(e')
     \scriptstyle'$}}\limits^{\kern0pt\raise.02ex
     \vbox to 0pt{\hbox{$\scriptscriptstyle\leftarrow$}\vss}}}\,}

\def\vsidash{{\mathop{\kern0pt s_i\kern-3.5pt\lower.3pt\hbox{${}
     \scriptstyle'$}}\limits^{\kern0pt\raise.02ex
     \vbox to 0pt{\hbox{$\scriptscriptstyle\rightarrow$}\vss}}}}

\def\vsp{\lowfwd {s_p}11}

\def\vsqdash{{\mathop{\kern0pt s_q\kern-3.5pt\lower.3pt\hbox{${}
     \scriptstyle'$}}\limits^{\kern0pt\raise.02ex
     \vbox to 0pt{\hbox{$\scriptscriptstyle\rightarrow$}\vss}}}}

\def\vS{{\hskip-1pt{\fwd S3}\hskip-1pt}} %{{\vec S}} %
\def\vR{{\hskip-1pt{\fwd R3}\hskip-1pt}} %{{\vec S}} %
\def\vSp{\lowfwd {S_p}11}

\def\vSstar{{\mathop{\kern0pt S\lower-1pt\hbox{$^*$}}\limits^{\kern2pt
     \vbox to 0pt{\hbox{$\scriptscriptstyle\rightarrow$}\vss}}}}
\def\vSdash{{\mathop{\kern0pt S\lower-1pt\hbox{${}% logically \v(e')
     \scriptstyle'$}}\limits^{\kern2pt\raise.1ex
     \vbox to 0pt{\hbox{$\scriptscriptstyle\rightarrow$}\vss}}}}
\def\vt{\lowfwd t{1.5}1}
\def\tv{\lowbkwd t{1.5}1}

\def\vO{\mathcal{O}}

\def\vOc{\overline{\mathcal{O}}}

\def\sub{\subseteq}

\newcommand\COMMENT[1]{}

\def\?#1{\vadjust{\vbox to 0pt{\vss\vskip-8pt\leftline{%
     \llap{\hbox{\vbox{\pretolerance=-1
     \doublehyphendemerits=0\finalhyphendemerits=0
     \hsize20truemm\tolerance=10000\small
     \lineskip=0pt\lineskiplimit=0pt
     \rightskip=0pt plus16truemm\baselineskip8pt\noindent
     \hskip0pt        %(without this, the first word is never hyphenated!)
     #1\endgraf}\hskip2truemm}}}\vss}}}

\newenvironment{txteq*}
  {
    \begin{equation*}
    \begin{minipage}[c]{0.85\textwidth} % set width to 0.9 x textwidth
    \em                                % switch on emph
  }
  {\end{minipage}\end{equation*}\ignorespacesafterend}

\title{Profinite tree sets}
 \author{Jakob Kneip} %
   
\renewcommand\S{\mathcal S}

\newcommand{\braces}[1]{\left(#1\right)}
\newcommand{\menge}[1]{\left\{#1\right\}}
\newcommand{\abs}[1]{\left |#1\right |}
\newcommand{\class}[1]{[#1]_D}

\newcommand{\tn}[1]{\textnormal{#1}}
\def\N{\mathbb{N}}
\newcommand{\lne}{\lneqq}

\def\lowfwd #1#2#3{{\mathop{\kern0pt #1}\limits^{\kern#2pt\raise.#3ex
\vbox to 0pt{\hbox{$\scriptscriptstyle\rightarrow$}\vss}}}}
\def\lowbkwd #1#2#3{{\mathop{\kern0pt #1}\limits^{\kern#2pt\raise.#3ex
\vbox to 0pt{\hbox{$\scriptscriptstyle\leftarrow$}\vss}}}}
\def\fwd #1#2{{\lowfwd{#1}{#2}{15}}}
\def\ve{\kern-1pt\lowfwd e{1.5}2\kern-1pt}
\def\ev{\kern-1pt\lowbkwd e{1.5}2\kern-1pt}
\def\vp{\lowfwd p{1.5}2}

\def\vr{\lowfwd r{1.5}2}
\def\rv{\lowbkwd r02}
\def\vu{\lowfwd u{1.5}2}
\def\uv{\lowbkwd u02}

\def\vrdash{{\mathop{\kern0pt r\lower.5pt\hbox{${}% logically \v(e')
     \scriptstyle'$}}\limits^{\kern0pt\raise.02ex
     \vbox to 0pt{\hbox{$\scriptscriptstyle\rightarrow$}\vss}}}}
\def\rvdash{{\mathop{\kern0pt r\lower.5pt\hbox{${}% logically \v(e')
     \scriptstyle'$}}\limits^{\kern0pt\raise.02ex
     \vbox to 0pt{\hbox{$\scriptscriptstyle\leftarrow$}\vss}}}}
\def\vrone{\lowfwd {r_1}12}

\def\vrtwo{\lowfwd {r_2}12}

\def\vs{\lowfwd s{1.5}1}
\def\sv{\lowbkwd s{1.5}1}

\def\vsidash{{\mathop{\kern0pt s_i\kern-3.5pt\lower.3pt\hbox{${}
     \scriptstyle'$}}\limits^{\kern0pt\raise.02ex
     \vbox to 0pt{\hbox{$\scriptscriptstyle\rightarrow$}\vss}}}}
\def\vS{{\hskip-1pt{\fwd S3}\hskip-1pt}} %{{\vec S}} %
\def\vSr{{\vec S}_{\raise.1ex\vbox to 0pt{\vss\hbox{$\scriptstyle\ge\vr$}}}}
\def\vSdash{{\mathop{\kern0pt S\lower-1pt\hbox{${}% logically \v(e')
     \scriptstyle'$}}\limits^{\kern2pt\raise.1ex
     \vbox to 0pt{\hbox{$\scriptscriptstyle\rightarrow$}\vss}}}}
\def\vsdash{{\mathop{\kern0pt s\lower.5pt\hbox{${}% logically \v(e')
     \scriptstyle'$}}\limits^{\kern0pt\raise.02ex
     \vbox to 0pt{\hbox{$\scriptscriptstyle\rightarrow$}\vss}}}}
     
\def\svdash{{\mathop{\kern0pt s\lower.5pt\hbox{${}% logically \v(e')
     \scriptstyle'$}}\limits^{\kern0pt\raise.02ex
     \vbox to 0pt{\hbox{$\scriptscriptstyle\leftarrow$}\vss}}}}
     
\def\vtdash{{\mathop{\kern0pt t\lower0pt\hbox{${}% logically \v(e')
     \scriptstyle'$}}\limits^{\kern0pt\raise.1ex
     \vbox to 0pt{\hbox{$\scriptscriptstyle\rightarrow$}\vss}}}}
     
\def\tvdash{{\mathop{\kern0pt t\lower0pt\hbox{${}% logically \v(e')
     \scriptstyle'$}}\limits^{\kern0pt\raise.1ex
     \vbox to 0pt{\hbox{$\scriptscriptstyle\leftarrow$}\vss}}}}
     
\def\vddash{{\mathop{\kern0pt d\raise1pt\hbox{${}% logically \v(e')
     \scriptstyle'$}}\limits^{\kern0pt\raise.02ex
     \vbox to 0pt{\hbox{$\scriptscriptstyle\rightarrow$}\vss}}}}
     
\def\dvdash{{\mathop{\kern0pt d\raise1pt\hbox{${}% logically \v(e')
     \scriptstyle'$}}\limits^{\kern0pt\raise.02ex
     \vbox to 0pt{\hbox{$\scriptscriptstyle\leftarrow$}\vss}}}}
     
\def\vtstar{{\mathop{\kern0pt t\raise2.5pt\hbox{${}% logically \v(e')
     \scriptstyle*$}}\limits^{\kern0pt\raise.1ex
     \vbox to 0pt{\hbox{$\scriptscriptstyle\rightarrow$}\vss}}}}
     
\def\tvstar{{\mathop{\kern0pt t\raise2.5pt\hbox{${}% logically \v(e')
     \scriptstyle*$}}\limits^{\kern0pt\raise.1ex
     \vbox to 0pt{\hbox{$\scriptscriptstyle\leftarrow$}\vss}}}}

\def\vtstarD{{\mathop{\kern0pt t\kern.5pt\raise3pt\hbox{${}% logically \v(e')
     \scriptstyle*$}{\kern-5.5pt\lower3pt\hbox{$
     \scriptstyle D$}}}\limits^{\kern0pt\raise.1ex
     \vbox to 0pt{\hbox{$\scriptscriptstyle\rightarrow$}\vss}}}}

\def\tvstarD{{\mathop{\kern0pt t\kern.5pt\raise3pt\hbox{${}% logically \v(e')
     \scriptstyle*$}{\kern-5.5pt\lower3pt\hbox{$
     \scriptstyle D$}}}\limits^{\kern0pt\raise.1ex
     \vbox to 0pt{\hbox{$\scriptscriptstyle\leftarrow$}\vss}}}}

\def\vt{\lowfwd t{1.5}1}
\def\tv{\lowbkwd t{1.5}1}

\def\vO{\mathcal{O}}

\def\vOc{\overline{\mathcal{O}}}
\def\vm{\lowfwd m{1.5}1}
\def\mv{\lowbkwd m{1.5}1}

\def\vt{\lowfwd t{1.5}1}
\def\tv{\lowbkwd t{1.5}1}

\def\vl{\lowfwd l{1.5}1}

\def\vx{\lowfwd x{1.5}1}
\def\xv{\lowbkwd x{1.5}1}
\def\vy{\lowfwd y{1.5}1}
\def\yv{\lowbkwd y{1.5}1}
\def\vz{\lowfwd z{1.5}1}

\def\vd{\lowfwd d{1.5}1}
\def\dv{\lowbkwd d{1.5}1}

\def\vb{\lowfwd b{1.5}1}

\def\vE{\lowfwd E{1.5}1}
\def\vM{\lowfwd M{1.5}1}
\def\Mv{\lowbkwd M{1.5}1}
\def\invlim{\varprojlim\,\,}
\def\D{\mathcal{D}}
\def\ph{\varphi}

\def\sube{\subseteq}

\newcommand{\dplus}[1]{D^+(#1)}
\newcommand{\dminus}[1]{D^-(#1)}

\begin{document}
\abovedisplayshortskip=-3pt plus3pt
\belowdisplayshortskip=6pt

\maketitle

\begin{abstract}\noindent
Tree sets are posets with additional structure that generalize tree-like objects in graphs, matroids, or other combinatorial structures. They are a special class of abstract separation systems.

We study infinite tree sets and how they relate to the finite tree sets they induce, and obtain a characterization of infinite tree sets in combinatorial terms.
\end{abstract}

\section{Introduction}\label{sec:intro}

This paper is a sequel to, and assumes familiarity with, two earlier papers~\cite{AbstractSepSys,ProfiniteASS}. The first of these~\cite{AbstractSepSys} introduced finite abstract separation systems, whereas the latter~\cite{ProfiniteASS} laid the foundations for extending the theory of separation systems to a broad class of infinite separation systems.

The theory of abstract separation systems introduced in~\cite{AbstractSepSys} aims to generalize the notion of tangles, a notion originally invented and studied by Robertson and Seymour in~\cite{GMX} to capture highly connected objects or regions in graphs, from graphs to other types of highly cohesive regions in graphs, matroids, or other combinatorial structures. The fundamental idea of Robertson and Seymour in~\cite{GMX} was to describe dense objects in graphs not directly, say by specifying a set of vertices, but indirectly, by having each low-order separation of the graph point towards that object. In contrast to specifying a list of vertices, this indirect approach allows one to capture objects or regions in the graph that are highly-connected in a global sense but not locally. A typical example for such a region is a large grid in a graph: since every vertex of a grid has low degree, the grid cannot be said to be locally highly connected. However, a large grid still constitutes a dense and highly cohesive structure in a graph, as witnessed by the fact that a large grid forces high tree-width.

For tangles in graphs, Robertson and Seymour~\cite{GMX} proved two fundamental types of theorems: a {\em tree-of-tangles theorem}, which shows how to find a tree-like shape in the graph which displays all the different tangles in that graph, and a {\em tangle-tree duality theorem}, which shows that if a graph has no tangles (of a particular order), then the entire graph can be cut by low-order separations into a tree-like structure witnessing this absence of tangles. Both of these types of theorems can be established in the framework of abstract separation systems. In fact, for the special case of separations of graphs, Robertson and Seymour's original version of these theorem can be obtained from the generalized abstract theorems~\cite{ProfilesNew,TangleTreeAbstract}.

In~\cite{ProfiniteASS} the foundations were laid for extending the tree-of-tangles theorem and the tangle-tree duality theorem to infinite separation systems: \cite{ProfiniteASS} introduced, and studied, separation systems that are profinite -- i.e. which are determined by the finite separation systems they induce. For this class of separation systems it is then possible to establish a tree-of-tangles theorem and a tangle-tree-duality theorem by applying and lifting their finite versions using compactness arguments~\cite{duality1inf}.

The central object in both the tree-of-tangles theorem as well as the tangle-tree duality theorem for abstract separation systems, apart from the tangles themselves, is the structure of a {\em tree set}, a nested separation system without any trivial elements: the tree-of-tangles theorem finds a tree set which distinguishes a given set of tangles, and the tangle-tree duality theorem finds a tree set witnessing that there are no tangles. The understanding of extensions of these theorems from finite to profinite separation systems thus necessitate a thorough understanding of the properties of profinite tree sets.

A recurring question for profinite separation systems is the following: if every induced finite subsystem has a certain property, does this property carry over to the profinite system -- and conversely, if the profinite separation system has a certain structure, can its induced finite subsystems be assumed to have that structure, too? In~\cite{ProfiniteASS} affirmative answers to both parts of this question were given for two of the most basic properties of separation systems: nestedness and regularity. A profinite separation system is nested as soon as all its finite subsystems are; and every nested profinite separation system can be obtained from suitable finite nested systems. For the second part a straightforward compactness argument is used to show that, in fact, all relevant separations of almost all finite subsystems are nested. The same assertion, and indeed the same general argument, holds for regularity, too~\cite[Proposition~5.6]{ProfiniteASS}.

In this paper we give a positive answer to the above question for the structural property of being a tree set, that is, being nested and containing no trivial separations (see~\cite{AbstractSepSys,ProfiniteASS} for formal definitions). Concretely, we show the following:

\begin{restatable}{THM}{ThmTreesets}\label{thm:treesets}
	{\em
	\tn{ }
	\begin{enumerate}
	\item[\tn{(i)}] Every inverse limit of finite tree sets is a tree set.
	\item[\tn{(ii)}] Every profinite tree set is an inverse limit of finite tree sets.
	\end{enumerate}
	}
\end{restatable}

The first part of Theorem~\ref{thm:treesets} is similar to the first part of Proposition~\cite[Proposition~5.6]{ProfiniteASS}. However, the second part of Theorem~\ref{thm:treesets} is much more difficult to show. The reason for this is that the compactness arguments used for nestedness and regularity do not work for tree sets: as~\cite[Example~5.7]{ProfiniteASS} shows it is possible that a separation is non-trivial in the profinite tree set, but that each of its induced finite separations is trivial in the respective finite subsystem. This makes it impossible to obtain the profinite tree set from finite tree sets by simply passing to suitable finite subsystems, as in the proof of~\cite[Proposition~5.6]{ProfiniteASS}.

In order to prove Theorem~\ref{thm:treesets} we will formulate a way of breaking up a given profinite tree~$ \tau $ set into finitely many parts in such a way that these parts form a finite tree set. The technical details of this are somewhat involved and are laid out in Section~\ref{sec:distinguishing}. Following that, in Section~\ref{sec:proof}, we will show that a carefully selected family of these finite tree sets can be used to re-obtain the profinite tree set~$ \tau $. To get this re-assembly of $ \tau $ to work we shall need to assume some that~$ \tau $ has certain structural properties; thus, to finish our proof of Theorem~\ref{thm:treesets}, we then need to verify that all profinite tree sets indeed have these structural properties. In doing so we will also obtain a characterization of the profinite tree sets in purely combinatorial terms.

Finally, in Section~\ref{sec:representations}, we apply the knowledge gained in the previous sections to extend certain representation theorems to profinite tree sets. In those theorems we seek to represent a tree set $ \tau $ as a separation system of bipartitions of a suitable groundset. This is easy to do in principle (see~\cite{TreeSets}), but becomes a challenging problem when one wants to minimize the groundset used for the representation.

\section{Separation Systems}\label{sec:separations}

For definitions and a basic discussion of abstract separation systems as well as of tree sets we refer the reader to \cite{AbstractSepSys} and \cite{TreeSets}. Additionally, we shall use the following terms.

We call a separation {\em co-small\/} if its inverse is small. If an oriented separation $\vr$ is trivial and this is witnessed by some separation~$s$, we also call the orientations $\vs,\sv$ of~$s$ {\em witnesses\/} of the triviality of~$\vr$. If $ \sigma $ is a splitting star of some separation system $ \vS $, and $ \sigma $ has size at least three, we call $ \sigma $ a {\em branching star} of $ \vS $ and its elements {\em branching points} of $ \vS $.

We shall be using the following two lemmas from~\cite{AbstractSepSys}:

\begin{LEM}[Extension Lemma]\tn{\cite[Lemma 4.1]{AbstractSepSys}}\label{lem:extension}
	Let $S$ be a set of unoriented separations, and let $P$ be a consistent partial orientation of $S$.
	\begin{enumerate}\itemsep=0pt
		\item[\tn{(i)}]$P$ extends to a consistent orientation~$ O $ of $S$ if and only if no element of $P$ is co-trivial in $S$.
		\item[\tn{(ii)}]If $\vp$ is maximal in $P$, then $O$ in \tn{(i)} can be chosen with $\vp$ maximal in $O$ if and only if $\vp$ is nontrivial in $\vS$.
		\item[\tn{(iii)}]If $S$ is nested, then the orientation~$ O $ in \tn{(ii)} is unique.
	\end{enumerate}
\end{LEM}

\goodbreak

Nested separation systems without degenerate or trivial elements are known as {\em tree sets\/}. A subset $\sigma\sub\vS$ {\em splits\/} a nested separation system~$\vS$ if $ \vS $ has a consistent orientation $ O $ such that $ O\subseteq\dcl(\sigma) $ and $ \sigma $ is precisely the set of maximal elements of~$ O $. Conversely, a consistent orientation $O$ of~$\vS$ {\em splits} (at~$\sigma$) if it is contained in the down-closure of the set~$\sigma$ of its maximal elements.

The consistent orientations of a finite nested separation system can be recovered from its splitting subsets by taking their down-closures, but infinite separation systems can have consistent orientations without any maximal elements.%
   \COMMENT{and which are, therefore, not the down-closure of their maximal elements, and hence not the down-closure of any splitting star}

\begin{LEM}\tn{\cite[Lemmas 4.4, 4.5]{AbstractSepSys}}\label{Remark8}
	The splitting subsets of a nested separation system $\vS$ without degenerate elements are proper stars. Their elements are neither trivial nor co-trivial in~$\vS$. If $S$ has a degenerate element~$s$, then $\{\vs\}$ is the unique splitting subset of~$\vS$.
\end{LEM}

The splitting stars of the edge tree set $\vE(T) $ of a tree~$ T$, for example, are the sets $\vec F_t$ of edges at a node~$t$ all oriented towards~$t$. By this correspondence, the nodes of~$T$ can be recovered from its edge tree set; if $T$ is finite, its nodes correspond bijectively to its consistent orientations.

Given two separation systems $ R,S $, a map $ f\colon \vR\to \vS $ is a {\em homomorphism} of separation systems if it commutes with their involutions and respects the ordering on~$\vR$. Formally, we say that $ f $ \textit{commutes with the involutions} of $\vR$ and~$\vS$ if $ \braces{f(\vr)}^*=f(\rv) $ for all ${\vr\in \vR}$. It \textit{respects the ordering} on~$\vR$ if $ f({\vrone})\le f({\vrtwo}) $ whenever $ {\vrone}\le{\vrtwo} $. Note that the condition for $ f $ to be order-respecting is not `if and only if': we allow that $ f({\vrone})\le f({\vrtwo}) $ also for incomparable~$ {\vrone},{\vrtwo}\in R $. Furthermore~$ f $ need not be injective. It can therefore happen that ${\vrone \le \vrtwo}$ with $ \vrone\ne\vrtwo $ but $f(\vrone) = f(\vrtwo)$, so $f$ need not preserve strict inequality. A~bijective homomorphism of separation systems whose inverse is also a homomorphism is an {\em isomorphism\/}.

We shall now prove two handy lemmas which provide sufficient conditions for a homomorphism of separation systems to be an isomorphism. These lemmas will be tailored towards their intended applications in Section~\ref{sec:proftreesets} but may be of some use in general.

As all trivial separations are small every regular nested separation system is a tree set. These two properties, regularity and nestedness, are preserved by homomorphisms of separations systems, albeit in different directions: the image of nested separations is nested, and the pre-image of regular separations is regular.

\begin{LEM}\label{beforeIso}
	Let $ f\colon \vR\to \vS $ be a homomorphism of separation systems. If $ S $ is regular then so is $ R $; and if $ R $ is nested then so is its image in~$ S $.
\end{LEM}

\begin{proof}
	First suppose some $ \vr\in \vR $ is small, that is $ \vr\le\rv $. But then
	\[ f(\vr)\le f(\rv)=\braces{f(\vr)}^*, \]
	so $ \vS $ contains a small element. Therefore if $ S $ is regular then $ R $ must be too.
	
	Now consider two unoriented separations $ s,s'\in S $. If there are $ r,r'\in R $ with $ s=f(r) $ and $ s'=f(r') $ and $ R $ is nested, then, say, $ \vr\le\vrdash $ and thus $ \vs:=f(\vr)\le f(\vrdash)=:\vsdash $. Hence if $ R $ is nested its image in $ S $ is nested too.
\end{proof}

%A bijection $ f\colon R\to S $ is an {\em isomorphism} of separation systems if both $ f $ and its inverse map are homomorphisms of separation systems. Two separation systems $ R,S $ are {\em isomorphic}, denoted as $ R\cong S $, if there is an isomorphism $ f\colon R\to S $ of separation systems. If one of $ R $ and $ S $ (and thus both) is a tree set we call $ f $ an {\em isomorphism of tree sets}.

Lemma~\ref{beforeIso} makes it possible to show that a homomorphism $ f\colon \vR\to \vS $ of separation systems is an isomorphism of tree sets without knowing beforehand that either $ R $ or $ S $ is a tree set:

\begin{LEM}\label{Isomorphism}
	Let $ f\colon \vR\to \vS $ be a bijective homomorphism of separation systems. If $ R $ is nested and $ S $ is regular then $ f $ is an isomorphism of tree sets.
\end{LEM}

\begin{proof}
	From Lemma~\ref{beforeIso} it follows that both $ R $ and $ S $ are regular and nested, which means they are regular tree sets. Therefore all we need to show is that the inverse of $ f $ is order-preserving, i.e. that $ \vr_1\le\vr_2 $ whenever $ f(\vr_1)\le f(\vr_2) $.
	Let $ \vr_1,\vr_2\in \vR $ with $ f(\vr_1)\le f(\vr_2) $ be given. Since $ R $ is nested $ r_1 $ and $ r_2 $ have comparable orientations.
	
	If $ \vr_1\ge\vr_2 $ then $ f(\vr_1)=f(\vr_2) $, implying $ \vr_1=\vr_2 $ and hence the claim.
	
	If $ \vr_1\le\rv_2 $ then $ f(\vr_1)\le f(\vr_2),f(\rv_2) $, contradicting the fact that $ S $ is a regular tree set.
	
	Finally if $ \vr_1\ge\rv_2 $ then $ f(\rv_2)\le f(\vr_2) $, contradicting the fact that $ S $ is regular.
	
	Hence $ \vr_1\le\vr_2 $, as desired.
\end{proof}

In our applications we sometimes already know that $ S $ is a tree set, but not that $ S $ is regular. The proof of Lemma~\ref{Isomorphism} still goes through though if we know that the pre-images of small separations are small:

\begin{LEM}\label{lem:isononreg}
	Let $ f\colon \vR\to \vS $ be a bijective homomorphism of separation systems. If $ R $ is nested, $ S $ is a tree set, and $ \vr\in \vR $ is small whenever $ f(\vr) $ is small, then $ f $ is an isomorphism of tree sets.
\end{LEM}

\begin{proof}
	It suffices to show that the inverse of $ f $ is order-preserving. Let $ \vr_1,\vr_2\in \vR $ with $ f(\vr_1)\le f(\vr_2) $ be given. If $ f(\vr_1)=f(\vr_2) $ or $ f(\rv_1)=f(\vr_2) $ we have $ \vr_1\le\vr_2 $ by assumption. Therefore we may assume that $ f(\vr_1)\lne f(\vr_2) $ and hence $ r_1\ne r_2 $. Since $ R $ is nested $ r_1 $ and $ r_2 $ have comparable orientations.
	
	If $ \vr_1\ge\vr_2 $ then $ f(\vr_1)=f(\vr_2) $ contradicting $ f(\vr_1)\lne f(\vr_2) $.
	
	If $ \vr_1\lne\rv_2 $ then $ f(\vr_1)\lne f(\vr_2),f(\rv_2) $, contradicting the fact that $ S $ contains no trivial element.
	
	Finally if $ \vr_1\ge\rv_2 $ then $ f(\rv_2)\lne f(\vr_1),f(\rv_2) $, again contradicting the fact that $ S $ contains no trivial element.
	
	Hence $ \vr_1\le\vr_2 $, as desired.
\end{proof}

Finally, we introduce the following non-standard notation: for separations $ \vr $ and $ \vs $ in some separation system we write $ \vr\lne\vs $ if $ \vr\le\vs $ and $ r\ne s $. Note that this is not the same as $ \vr<\vs $: if $ \vr<\vs $, that is, if $ \vr\le\vs $ and $ \vr $ and $ \vs $ differ as {\em oriented} separations, then $ r $ and $ s $ could still be the same {\em unoriented} separation if $ \vr=\sv $. On the other hand, if $ \vr\lne\vs $, then $ r $ and $ s $ must be distinct as {\em unoriented} separations. In nearly all places in this paper in which we consider separations $ \vr $ and $ \vs $ with $ \vr\le\vs $, and $ \vr $ and $ \vs $ are different oriented separations, we shall also want $ r $ and $ s $ to be different as unoriented separations and hence will write $ \vr\lne\vs $. To avoid confusion we shall, from now on, never use the symbol `$ < $' again.

\section{Profinite tree sets}\label{sec:proftreesets}

\subsection{Introduction}

We refer the reader to~\cite{ProfiniteASS} for an introduction to inverse limits of sets, inverse systems of separation systems, and profinite separation systems. In particular we assume familiarity with Section~3, Section~4, and Section~5 up to Example~5.7 of~\cite{ProfiniteASS}; we shall follow the terms and notation given there.

Let $ \S=(\vSp\mid p\in P) $ be an inverse system of finite separation systems and $ \vS=\invlim\S $ its inverse limit. The separations in $ \vS $ are of the form $ \vs=(\vsp\mid p\in P) $ with $ \vsp\in\vSp $. To enhance readability, when no confusion is possible, we will simply write $ \vs\in\vS $ for separations in $ \vS $ and implicitly assume that $ \vs=(\vsp\mid p\in P) $. Thus, if no context is given, $ \vsp $ will always be the projection of $ \vs $ to $ \vSp $.

We are interested in the relations between properties of the $ \vSp $ and the properties of~$ \vS $: which structural properties of separation systems can be `projected downwards' from $ \vS $, and which can be `lifted upwards' from the $ \vSp $ to $ \vS $? In~\cite{ProfiniteASS} this question was answered for the two properties of being nested, and being regular: both of these properties `lift up' in the sense that if all $ \vSp $ are nested (resp. regular), then $ \vS $ is nested (resp. regular), too. Moreover, both of these properties also `project downwards': if $ \vS $ is nested (resp. regular), then the $ \vSp $ can be assumed to be nested (resp. regular), too. More precisely: every nested (resp. regular) profinite separation system is the inverse limit of nested (resp. regular) finite separation systems.

Indeed, the following was shown in~\cite{ProfiniteASS}:

\begin{PROP}[\cite{ProfiniteASS}]\label{prop:regular}
	{\em
		\tn{ }
   \vskip-\medskipamount\vskip0pt
		\begin{enumerate}\itemsep=0pt
			\item[\tn{(i)}] Every inverse limit of finite regular separation systems is regular.
			\item[\tn{(ii)}] Every profinite regular separation system is an inverse limit of finite regular separation systems.
		\end{enumerate}
	}
\end{PROP}

The same assertion holds with `regular' being replaced by `nested' and follows from~\cite[Lemma~5.4]{ProfiniteASS}.

The proof of Proposition~\ref{prop:regular} is straightforward: (i) follows directly from the definition of an inverse limit of separation systems. For (ii), one observes that if every $ \vSp $ in an inverse system $ \S=(\vSp\mid p\in P) $ contains a small separation, then these small separation lift to a common element of $ \vS=\invlim\S $, which will be small, too. Therefore if $ \vS $ is regular then $ \S $ must already contain a sub-system of regular finite separation systems whose inverse limit is $ \vS $.

However, already in~\cite{ProfiniteASS} it was observed that the same is not true for trivial separations: it is possible that some $ \vs=(\vsp\mid p\in P)\in\vS $ is non-trivial in $ \vS $, but its projections $ \vsp\in\vSp $ are trivial in $ \vSp $ for all $ p\in P $. See~\cite[Example~5.7]{ProfiniteASS} for an example of this behaviour. The problem with these so-called {\em finitely trivial} separations is that the witnesses of the triviality of $ \vsp $ in $ \vSp $ may not lift to a witness of the triviality of $ \vs $ in $ \vS $.

The aim of this section is to overcome the difficulties laid out above and establish the following theorem:

\ThmTreesets*

As~\cite[Example~5.7]{ProfiniteASS} shows, it is not possible to prove (ii) of Theorem~\ref{thm:treesets} with a direct compactness argument. Therefore a new approach is needed in order to establish Theorem~\ref{thm:treesets}. Before we get to this we shall briefly deal with the much easier special case of regular tree sets:

\begin{THM}\label{thm:treesetsregular}
	{\em
		\tn{ }
		\begin{enumerate}
			\item[\tn{(i)}] Every inverse limit of regular finite tree sets is a regular tree set.
			\item[\tn{(ii)}] Every regular profinite tree set is an inverse limit of regular finite tree sets.
		\end{enumerate}
	}
\end{THM}

\begin{proof}
	Assertion \tn{(i)} follows directly from~\cite[Lemma~5.4]{ProfiniteASS}, \cite[Proposition~5.6(i)]{ProfiniteASS} and the fact that every regular nested separation system is a regular tree set as all trivial separations are small.
	
	For \tn{(ii)} let $ \vS $ be a regular tree set with $ \vS=\invlim\S $ for an inverse system $ \S=(\vSp\mid p\in P) $. We may assume without loss of generality that $ \S $ is surjective; then every $ \vSp $ is nested. Furthermore as seen in the proof of \cite[Proposition~5.6(ii)]{ProfiniteASS} there exists some $ p_0\in P $ for which $ \vS_{p_0} $ is regular. Now $ (\vSp\mid p\ge p_0) $ is the desired inverse system of finite regular tree sets.	
\end{proof}

Let us now return to Theorem~\ref{thm:treesets}. Its first part is straightforward to prove:\\
\\
{\em Proof of Theorem~\ref{thm:treesets}}(i). Let $ \S=(\vSp\mid p\in P) $ be an inverse system of finite tree sets and $ \vS=\invlim\S $. Suppose some $ \vs=(\vs_p\mid p\in P)\in\vS $ is trivial in $ \vS $ with witness~$ r $. Since $ P $ is a directed set there is a $ p\in P $ such that $ \vs_p\ne\vr_p $ and~$ \vs_p\ne\rv_p $. But $ \vs\lne\vr,\rv $ in $ \vS $ implies that~$ \vs_p\lne\vr_p,\rv_p $ in~$ S_p $, contrary to the assumption that $ S_p $ is a tree set.

Moreover if every $ S_p $ is nested then so is~$ \vS $ by~\cite[Lemma~5.4]{ProfiniteASS}.

Therefore $ \vS $ is a tree set.\hfill$ \Box $\\
\\
We will postpone the proof of (ii) until the end of Section~\ref{sec:proof}. Our approach shall be to decompose a given tree set $ \tau $ into finitely many parts which together form a finite tree set. The family of all these `quotients' of $ \tau $ should then form an inverse system whose inverse limit is precisely~$ \tau $. However the exact definition of these decompositions is fairly technical, and we introduce it in the next section.

\subsection{Distinguishing separations}\label{sec:distinguishing}

This section lays the technical foundations for the proof of Theorem~\ref{thm:treesets}. Given a tree set~$ \tau $ our aim is to find a way of defining finite quotients of~$ \tau $ that form an inverse system of finite tree sets whose inverse limit is isomorphic to~$ \tau $. The latter part of this will be done in the next two sections, while in this section we define these `finite quotients' and analyse their properties.

To this end for any finite set of stars in~$ \tau $ we define an equivalence relation on~$ \tau $ which essentially breaks up~$ \tau $ into finitely many chunks. After proving a few basic facts about this equivalence relation we find certain conditions that ensure that the equivalence classes of~$ \tau $ form a finite tree set, as needed for the proof of Theorem~\ref{thm:treesets}(ii).

Following this main part of the section we analyse these equivalence relations a bit more and prove a few key lemmas.\\

Let~$ \tau $ be a tree set. A {\em selection} of~$ \tau $ is a non-empty finite set $ D\sube\tau $ of oriented separations with $ \abs{\sigma\cap D}\ne 1 $ for every splitting star $ \sigma $ of~$ \tau $.

Let us show that any selection $ D $ of~$ \tau $ divides~$ \tau $ into different sections between the stars that meet~$ D $. We make this precise by defining an equivalence relation~$ \sim_D $ on~$ \tau $.

Recall that, for $ \vr,\vs\in\tau $, we write $ \vr\lne\vs $ as shorthand notation for $ \vr\le\vs $ and $ r\ne s $.

For a separation $ \vs\in\tau $ and a selection $ D $ let
\[ \dplus{\vs}:=\menge{\vd\in D\mid \vd\lneqq\vs} \]
and
\[ \dminus{\vs}:=\dplus{\sv}\,. \]
Two separations $ \vs,\vr $ are $ D $-{\em equivalent} for a selection $ D $, denoted as $ \vs\sim_D\vr $, if $ \dplus{\vs}=\dplus{\vr} $ and $ \dminus{\vs}=\dminus{\vr} $. This is an equivalence relation with finitely many classes. We write $ \class{\vs} $ for the equivalence class of $ \vs\in\tau $ under~$ \sim_D $. A separation $ \vd\in D $ {\em distinguishes} $ \vr $ and $ \vs $ if
\[ \vd\in\braces{\dplus{\vr}}\triangle\braces{\dplus{\vs}}\tn{  or  }\vd\in\braces{\dminus{\vr}}\triangle\braces{\dminus{\vs}}. \]
Thus $ \vr $ and $ \vs $ are $ D $-equivalent if and only if no $ \vd\in D  $ distinguishes them.\\
\\
For a selection $ D $ of~$ \tau $ and separations $ \vr\le\vs $ it follows from the definitions that $ \dplus{\vr}\subseteq\dplus{\vs} $ and $ \dminus{\vr}\supseteq\dminus{\vs} $. Furthermore $ \dplus{\vs}\cap\dminus{\vs}=\emptyset $ for all $ \vs\in\tau $ as any element of this intersection would be trivial with witness~$ \vs $. This implies that $ \vs\sim_D\sv $ if and only if $ \dplus{\vs}=\dminus{\vs}=\emptyset $. But $ \dplus{\vs}\cup\dminus{\vs} $ is never empty and in fact contains an element of every splitting star that meets $ D $, so $ \vs\not\sim_D\sv $ for every $ \vs\in\tau $.\\
\\
The next lemma shows a few basic properties of~$ \sim_D $. The first of these is especially important, as it will enable us to turn the equivalence classes of $ \sim_D $ on $ \tau $ into a separation system.

\begin{LEM}\label{dist_list}
Let~$ \tau $ be a tree set, $ D $ a selection and~$ \vr,\vs,\vt\in\tau $.
\begin{enumerate}
\item[\tn{(i)}] If $ \vr\sim_D\vs $ then $ \rv\sim_D\sv $.
\item[\tn{(ii)}] If $ \vr\le\vs\le\vt $ and $ \vr\sim_D\vt $ then~$ \vr\sim_D\vs\sim_D\vt $.
\item[\tn{(iii)}] If there is no $ \vd\in D  $ with $ \vd\lneqq\vs $, and $ \vr\le\vs $, then~$ \vr\sim_D\vs $.
\item[\tn{(iv)}] If $ \vr\le\vs\le\tv $ and $ \vs\sim_D\vt $ then~$ \vr\sim_D\vs $.
\item[\tn{(v)}] If $ \vr\ge\vs\ge\tv $ and $ \vs\sim_D\vt $ then~$ \vr\sim_D\vs $.
\end{enumerate}
\end{LEM}

\begin{proof}
$ \tn{(i)} $ This follows from
\[ \dplus{\sv}=\dminus{\vs}=\dminus{\vr}=\dplus{\rv} \]
and
\[ \dminus{\sv}=\dplus{\vs}=\dplus{\vr}=\dminus{\rv}. \]\\
\\
$ \tn{(ii)} $ By the observation above
\[ \dplus{\vr}\subseteq\dplus{\vs}\subseteq\dplus{\vt}=\dplus{\vr} \]
and similarly $ \dminus{\vr}=\dminus{\vs} $, hence $ \vr\sim_D\vs $.\\
\\
$ \tn{(iii)} $ By assumption $ \dplus{\vr}=\dplus{\vs}=\emptyset $. Furthermore $ \dminus{\vs}\subseteq\dminus{\vr} $. Suppose there is a $ \vd\in D  $ with $ \vd\lneqq\rv $ but not $ \vd\lneqq\sv $. Then $ \dv\lne\vs $ as $ \vd\not\lne\vs $ by assumption. Let $ \sigma $ be the splitting star containing $ \vd $ and $ \ve\in\sigma\cap D $ with $ \vd\ne\ve $. Then $ \ve\le\dv\lne\vs $ by the star property, contradicting the assumption that there is no such~$ \ve\in D $.\\
\\
$ \tn{(iv)} $ As $ \dplus{\vs}\subseteq\dplus{\tv}=\dplus{\sv} $ there can be no $ \vd\in D  $ with $ \vd\lneqq\vs $ as it would be trivial with witness $ s $. Thus $ \vs\sim_D\vr $ by (iii).\\
\\
$ \tn{(v)} $ This follows from (i) and (iv).
\end{proof}

We now define precisely in which way we want to turn the equivalence classes of~$ \tau $ into a separation system.

Let $ D $ be a selection of a tree set~$ \tau $. Then write
\[ \tau/D:=\braces{\menge{\class{\vs}\mid \vs\in\tau},\,\le\,,\,*\,}  \]
with $ (\class{\vs})^*:=\class{\sv} $ and $ \class{\vs}\le\class{\vr} $ if there are~$ \vsdash\in\class{\vs} $ and~$ \vrdash\in\class{\vr} $ with~$ \vsdash\le\vrdash $.

If for some selection $ D $ the relation $ \le $ of $ \tau/D $ is a partial order then~$ \tau/D $ is a separation system by Lemma~\ref{dist_list}(i). In that case $ \tau/D $ would even be nested because~$ \tau $ is.

Our aim is to ensure that $ \tau/D $ is a tree set. For this we first find sufficient conditions for $ \le $ to be a partial order, and then show that these conditions are strong enough to ensure that $ \tau/D $ does not contain any trivial elements.

The relation $ \le $ on $ \tau/D $ is reflexive by definition, thus we need to show that it is transitive and anti-symmetric. For the latter no further assumptions are needed, so we begin by proving the anti-symmetry.

\begin{LEM}\label{anti-symmetric}
Let~$ \tau $ be a tree set and $ D $ a selection. If $ \vr\le\vx $ and $ \vs\ge\vy $ for $ \vr,\vs,\vx,\vy\in\tau $ with $ \vr\sim_D\vs $ and $ \vx\sim_D\vy $ then also~$ \vr\sim_D\vx $.
\end{LEM}

\begin{proof}
We have
\[ \dplus{\vr}\sube\dplus{\vx}=\dplus{\vx}\sube\dplus{\vs}=\dplus{\vr} \]
and similarly $ \dminus{\vr}=\dminus{\vx} $.
\end{proof}

This shows that $ \le $ on $ \tau/D $ is antisymmetric. To prove transitivity we need further assumptions, as the following example demonstrates.

\begin{EX}\label{ex:non-trans}
Let $ T $ be the following graph.

%TODO: Make prettier pictures.

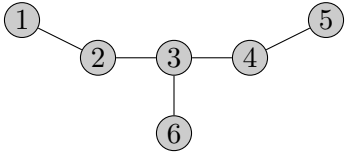
\begin{center}

\begin{tikzpicture}[scale=0.5]
\tikzstyle{every node}=[circle, draw, fill=black!20, inner sep=1.5pt, minimum width=0.7pt]
	\node (1) at (0,3) {1};
	\node (2) at (2,2) {2};
	\node (3) at (4,2) {3};
	\node (4) at (6,2) {4};
	\node (5) at (8,3) {5};
	\node (6) at (4,0) {6};

	\draw (1) -- (2) -- (3) -- (4) -- (5);
	\draw (3) -- (6);
%	\caption{The graph~$ G $.}
\end{tikzpicture}

\end{center}

The edge tree set $ \tau(T) $ (as defined in~\cite{TreeSets}) is regular and we have $ (6,3)\le(3,2) $ and $ (4,3)\le(3,6) $. For the selection $ D=\menge{(1,2),(3,2),(3,4),(5,4)} $ the edges $ (2,3) $ and $ (3,4) $ get identified in $ \tau(T)/D $, implying
\[ \class{(6,3)}\le\class{(3,2)}=\class{(4,3)}\le\class{(3,6)} \]
in $ \tau(T)/D $. But $ (6,3)\not\le(3,6) $, so $ \le $ is not transitive on $ \tau(T)/D $.
\end{EX}

This example exploits the fact that there is a branching star (i.e., a splitting star of size at least three) that does not meet $ D $ between two splitting stars that do meet~$ D $. In order to prevent this counterexample to transitivity one could ask that $ D $ meets every branching star that lies between two elements of~$ D $. But this alone is not enough to ensure that $ \le $ on $ \tau/D $ is transitive: if we replace the separation $ (6,3) $ in Example~\ref{ex:non-trans} above with a chain of order type $ \omega $ the resulting tree set would not have any branching stars, but the transitivity of $ \le $ would still fail for the same reason. Therefore we also need an assumption on~$ \tau $ that ensures that whenever there is a three-star as in the example above we can also find a branching star, which would then be subject to the condition on~$ D $.

Recall that $ \vb\in\tau $ is a branching point of~$ \tau $ if $ \vb $ lies in a splitting star of size at least three.

Call a selection $ D $ of a tree set~$ \tau $ {\em branch-closed} if $ \vb\in D  $ for every branching point $ \vb $ of~$ \tau $ for which there are $ \vd_1,\vd_2\in D $ with $ \vd_1\le\vb\le\dv_2 $. Furthermore~$ \tau $ is {\em chain-complete} if every non-empty\footnote{This differs from the common definition of a chain-complete poset, which usually omits the `non-empty', as it does not imply that~$ \tau $ has a smallest element.} chain $ C\subseteq\tau $ has a supremum in~$ \tau $.

Let us see an example of a chain-complete tree set:

\begin{EX}\label{ex:(non)chain-complete}
Let $ X $ be a non-empty set of positive real numbers and let
\[ \tau(X):=\braces{\bigcup_{x\in X}\menge{x,-x},\,\curlyle\,,\,*}, \]
where $ x^*:=-x $, and $ x\curlyle y $ if and only if $ x\le y $ as real numbers and $ x $ and $ y $ have the same sign. Then $ \tau(X) $ is a tree set. Moreover, $ \tau(X) $ is chain-complete if and only if $ X $ is compact as a subset of $ \mathbb{R} $.
\end{EX}

We claim that the two conditions that $ \tau $ is chain-complete and $ D $ branch-closed are enough to ensure that $ \le $ on $ \tau/D $ is transitive and hence a partial order. Before we prove this claim we need to establish some basic properties of chain-complete tree sets, beginning with the fact that every chain has not only a supremum but an infimum too:

\begin{LEM}
Let~$ \tau $ be a chain-complete tree set and $ C $ a chain. Then $ C $ has an infimum in~$ \tau $.
\end{LEM}

\begin{proof}
Consider the chain
\[ C':=\menge{\tv\mid \vt\in C} \]
and let $ \sv $ be its supremum in~$ \tau $. Then $ \vs $ is the infimum of $ C $.
\end{proof}

Recall that an orientation~$ O $ of~$ \tau $ is splitting if every element of $ O $ lies below some maximal element of~$ O $. 

The usual way to find a branching star in a tree set~$ \tau $ is to define a consistent orientation with three or more maximal elements and then show that it is splitting. The first part can be done with the Extension Lemma (Lemma~\ref{lem:extension}). For the latter part the following lemma provides a sufficient condition for a consistent orientation to be splitting. It turns out that having {\em two} maximal elements is already enough, if~$ \tau $ is chain-complete:

\begin{LEM}\label{lem:twoclosedgeneral}
Let $\tau$ be a chain-complete tree set and $O$ a consistent orientation of $\tau$ with two or more maximal elements. Then $ O $ is splitting.
\end{LEM}

\begin{proof}
Let $\vr,\vs$ be two maximal elements of $O$ and let $\vt\in O$ be any separation. If $\vt$ lies below $\vr$ or $\vs$ there is nothing to show. If not consider the up-closure $C:=\ucl(\vt)\sub O$ of $\vt$ in~$O$. As $O$ is consistent $C$ is a chain, which has a supremum $\vm\in\tau$ by assumption. $\rv$ and $\sv$ are upper bounds for $C$, so~$\mv\ge\vr,\vs$. But $ \vr $ and $ \vs $ are maximal in $O$, implying $\mv\notin O$ and hence~$\vm\in O$.
\end{proof}

With Lemma~\ref{lem:twoclosedgeneral} we can now show that if we have a three-star in a chain-complete tree set we can find a branching star `in the same location':

\begin{PROP}\label{star-push}
Let~$ \tau $ be a chain-complete tree set and $ \sigma $ a star with exactly three elements. Then there is a unique branching star $ \sigma' $ of~$ \tau $ such that every element of $ \sigma $ lies below a different element of~$ \sigma' $.
\end{PROP}

\begin{proof}
Let $ \sigma=\{\vr,\vs,\vt\} $ and
\[R=\menge{\vx\in\tau\mid \vr\le\vx\le\sv\tn{ and }\vr\le\vx\le\tv}.\]
Then $R$ is a chain and by assumption $\vrdash=\sup R$ exists. As $\sv$ and $\tv$ are lower bounds for $R$ we have~$\vrdash\in R$. Similarly define
\[S=\menge{\vx\in\tau\mid \vs\le\vx\le\rv\tn{ and }\vs\le\vx\le\tv}\]
and
\[T=\menge{\vx\in\tau\mid \vt\le\vx\le\rv\tn{ and }\vt\le\vx\le\sv}\]
as well as $\vsdash=\sup S$ and~$\vtdash=\sup T$. Since $R,S$ and $T$ are disjoint $\{\vrdash,\vsdash,\vtdash\}$ forms a three-star. Observe that there is no $\vx\in\tau$ that lies strictly between $\vrdash$ and $\svdash$: for $ \vx\in\tau $ with $ \vrdash\le\vx\le\svdash $ then neither $ \vx\le\vt $ nor $ \vx\ge\vt $ since that would make $ r $ or $ s $ trivial, respectively. Thus either $ \vx\le\tv $ or $ \vx\ge\vt $, giving $ \vx\in R $ or $ \xv\in S $, respectively, and thus equality with $ \vrdash $ or $ \svdash $ by their maximality.

Similarly, there are no separations strictly between $\vsdash$ and $\tvdash$, or between $\vtdash$ and $\rvdash$.

Applying Lemma~\ref{lem:extension} to $ \{\vrdash,\vsdash,\vtdash\} $ yields a consistent orientation~$O$ in which $\vrdash$ is maximal. By the above observation $\vsdash$ and $\vtdash$ are maximal in $O$ too. It follows from Lemma~\ref{lem:twoclosedgeneral} that $ O $ is splitting, so the set $ \sigma' $ of its maximal elements is the desired branching star.

The uniqueness follows from the fact that if $ \sigma_1 $ and $ \sigma_2 $ are two distinct splitting stars, there is a $ \vs\in\sigma_2 $ which is an upper bound for all elements of~$ \sigma_1 $ but one. Hence if three separations lie below different elements of $ \sigma_1 $, at least two of them will lie below the same element of $ \sigma_2 $.
\end{proof}

This proposition is useful as it often allows us to work with splitting stars without loss of generality in the context of selections. We will also use it in the next section, especially the uniqueness part which is not important in this section.

We now show in three steps that $ \le $ on $ \tau/D $ is transitive for branch-closed $ D $ and chain-complete~$ \tau $. First we show that if a counterexample to the transitivity exists it must be a three-star with one element equivalent to the inverse of the second, as in Example~\ref{ex:non-trans}. Then we apply Proposition~\ref{star-push} to this three-star to obtain a branching star, of which we show that it is still a counterexample to the transitivity. Finally we derive a contradiction to the assumption that $ D $ is branch-closed.

\begin{LEM}\label{non-trans1}
Let~$ \tau $ be a tree set and $ D $ a selection. If $ \le $ on $ \tau/D $ is not transitive then there is a three-star $ \{\vr,\vs_1,\sv_2\} $ such that $ \vs_1\sim_D\vs_2 $ but neither $ \vr\sim_D\vs_1 $ nor~$ \vr\sim_D\sv_1 $.
\end{LEM}

\begin{proof}
Suppose there are $ \class{\vx},\class{\vy},\class{\vz}\in\tau/D $ such that $ \class{\vx}\le\class{\vy} $ and $ \class{\vy}\le\class{\vz} $ but $ \class{\vx}\not\le\class{\vz} $. Pick $ \vr\in\class{\vx},\vs_1,\vs_2\in\class{\vy} $ and $ \vt\in\class{\vz} $ with $ \vr\le\vs_1 $ and $ \vs_2\le\vt $.

At most one of $ \vr $ and $ \vt $ can be $ D $-equivalent to $ \sv_1 $ by assumption. Suppose that $ \vr\not\sim_D\sv_1 $ (the case $ \vt\not\sim_D\sv_1 $ is symmetrical).

Because~$ \tau $ is a tree set $ s_1 $ and $ s_2 $ have comparable orientations. If $ \vs_1\le\vs_2 $ then $ \vr\le\vt $, contradicting $ \class{\vx}\not\le\class{\vz} $. By Lemma~\ref{dist_list}(iv) and (v) $ \vs_1\not\le\sv_2 $ and $ \vs_1\not\ge\sv_2 $ as $ \vr\not\sim_D\vs_1,\sv_1 $. Hence $ \vs_1\ge\vs_2 $. Furthermore $ \vr\not\le\vs_2 $ as $ \vr\not\le\vt $, and $ \vr\not\ge\vs_2 $ by Lemma~\ref{dist_list}(i), so $ \{\vr,\sv_1,\vs_2\} $ must be a three-star.
\end{proof}

This completes the first of the three steps. In the next step we show that if a counterexample to the transitivity of $ \le $ exists there is a counterexample which is a branching star.

\begin{LEM}\label{non-trans2}
Let~$ \tau $ be a chain-complete tree set and $ D $ a selection. If $ \le $ on $ \tau/D $ is not transitive then there is a three-star $ \{\vr,\vs_1,\sv_2\} $ of branching points such that $ \vs_1\sim_D\vs_2 $ but neither $ \vr\sim_D\vs_1 $ nor~$\vr\sim_D\sv_1 $.
\end{LEM}

\begin{proof}
By Lemma~\ref{non-trans1} there is a three-star $ \{\vx,\vy_1,\yv_2\} $ with $ \vy_1\sim_D\vy_2 $ and $ \vx\not\sim_D\vy_1,\yv_1 $. An application of Proposition~\ref{star-push} yields a branching star $ \sigma $ with a three-star $ \{\vr,\vs_1,\sv_2\}\sube\sigma $ for which $ \vx\le\vr $ and~$ \vy_1\le\vs_1\le\vs_2\le\vy_2 $. From Lemma~\ref{dist_list}(ii) it follows that~$ \vy_1\sim_D\vs_1\sim_D\vs_2\sim_D\vy_2 $. Furthermore Lemma~\ref{dist_list}(i) and (iv) imply that~$ \vr\not\sim_D\vs_1 $ and $ \vr\not\sim_D\sv_1 $, as otherwise $ \vx\sim_D~\vy_1 $ or $ \vx\sim_D\yv_1 $ contrary to assumption. Thus $ \{\vr,\vs_1,\sv_2\} $ is the desired three-star.
\end{proof}

For the third step we need to show that there are elements of $ D $ that allow us to apply the branch-closedness of~$ D $ to derive a contradiction.

\begin{LEM}\label{non-trans3}
Let~$ \tau $ be a chain-complete tree set and $ D $ a selection. Let $ \{\vr,\vs_1,\sv_2\} $ be a three-star in~$ \tau $ with $ \vs_1\sim_D\vs_2 $ but neither $ \vr\sim_D\vs_1 $ nor $ \vr\sim_D\sv_1 $. Then there are $ \vd_1,\vd_2\in D  $ with $ \vd_1\lne\vs_1 $ and $ \dv_2\lne\sv_2 $.
\end{LEM}

\begin{proof}
Let $ \vd\in D $ distinguish $ \vr $ and~$ \vs_1 $; we will show that~$ \vd\lne\vs_1 $. This $ \vd $ cannot lie in $ \dplus{\vr}\setminus\dplus{\vs_1} $ as then it would also distinguish $ \vs_1 $ and~$ \vs_2 $. Furthermore $ \dminus{\vs_1}=\dminus{\vs_2}\sube\dminus{\vr} $ by the star property, so $ \vd $ cannot lie in $ \dminus{\vs_1}\setminus\dminus{\vr} $ either. If $ \vd\in\dminus{\vr}\setminus\dminus{\vs_1} $ then any $ \ve\in\sigma\cap D $ with $ \ve\ne\vd $ would distinguish $ \vs_1 $ and $ \vs_2 $, where $ \sigma $ is the splitting star containing~$ \vd $. Therefore $ \vd\in\dplus{\vs_1}\setminus\dplus{\vr} $, so in particular~$ \vd\lne\vs_1 $.

Repeating this argument for a $ \vd_2\in D $ that distinguishes $ \vr $ and $ \sv_2 $ shows $ \dv_2\lne\sv_2 $ and hence the claim.
\end{proof}

Finally we put the above lemmas together to prove that $ \le $ in $ \tau/D $ is transitive.

\begin{LEM}\label{dist_transitive}
Let~$ \tau $ be a chain-complete tree set and $ D $ a branch-closed selection. Then $ \le $ on $ \tau/D $ is transitive.
\end{LEM}

\begin{proof}
Suppose $ \le $ is not transitive. Then by Lemma~\ref{non-trans2} there is a three-star $ \{\vr,\vs_1,\sv_2\} $ of branching points such that $ \vs_1\sim_D\vs_2 $ but neither $ \vr\sim_D\vs_1 $ nor~$ \vr\sim_D\sv_1 $. Applying Lemma~\ref{non-trans3} to this star yields $ \vd_1,\vd_2\in D $ with~$ \vd_1\lne\vs_1\le\vs_2\lne\vd_2 $. As $ D $ is branch-closed and $ \vs_1 $ a branching point this implies $ \vs_1\in D $; but then $ \vs_1\in\dplus{\vs_2}\setminus\dplus{\vs_1} $, contrary to the assumption that~$ \vs_1\sim_D\vs_2 $.

Hence $ \le $ is transitive as claimed.
\end{proof}

Therefore $ \le $ on $ \tau/D $ is a partial order, so $ \tau/D $ is a nested separation system for chain-complete~$ \tau $ and branch-closed~$ D $. To prove that $ \tau/D $ is a tree set it is thus left to show that it does not contain any trivial elements.

The next example shows that $ \tau/D $ may well contain a trivial element even in cases where $ \le $ is a partial order. However, this too exploits that $ D $ is not branch-closed, and we will subsequently prove that $ \tau/D $ is indeed a tree set for branch-closed~$ D $.

\begin{EX}\label{ex:trivial}
Let $ T $ be the following graph.

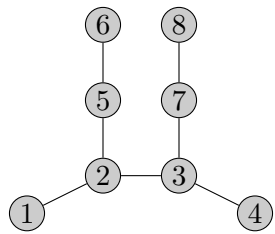
\begin{center}

\begin{tikzpicture}[scale=0.5]
\tikzstyle{every node}=[circle, draw, fill=black!20, inner sep=1.5pt, minimum width=0.7pt]
	\node (1) at (0,0) {1};
	\node (2) at (2,1) {2};
	\node (3) at (4,1) {3};
	\node (4) at (6,0) {4};
	\node (5) at (2,3) {5};
	\node (6) at (2,5) {6};
	\node (7) at (4,3) {7};
	\node (8) at (4,5) {8};

	\draw (1) -- (2) -- (3) -- (4);
	\draw (2) -- (5) -- (6);
	\draw (3) -- (7) -- (8);
%	\caption{The graph~$ H $.}
\end{tikzpicture}

\end{center}

The edge tree set $ \tau(T) $ is regular. For the selection $ D=\menge{(2,5),(6,5),(3,7),(8,7)} $ we have $ (1,2)\sim_D(4,3) $ and thus $ \class{(1,2)}\le\class{(2,3)},\class{(3,2)} $. As $ D $ distinguishes $ (1,2) $ from $ (2,3) $ and from $ (3,2) $ this means that $ \class{(1,2)} $ is trivial in~$ \tau(T)/D $.
\end{EX}

The proof that $ \tau/D $ has no trivial elements if~$ \tau $ is chain-complete and $ D $ is branch-closed will again be carried out in multiple steps. First we show that the configuration from Example~\ref{ex:trivial} is the only possible type of counterexample. Following that we prove that if this counterexample occurs there are elements of $ D $ we can use to apply the branch-closedness of $ D $ with.

\begin{LEM}\label{non-triv1}
Let~$ \tau $ be a chain-complete tree set and $ D $ a branch-closed selection. If $ \tau/D $ contains a trivial element then there are $ \vr,\vs,\vx\in\tau $ with $ \vr\le\vx\le\sv $ and $ \vr\sim_D\vs $ but neither $ \vr\sim_D\vx $ nor~$ \vr\sim_D\xv $.
\end{LEM}

\begin{proof}
If $ \tau/D $ contains a trivial element then there are $ \vr,\vx\in\tau $ with $ \vr\le\vx $ and $ \class{\vr}\lne\class{\vx},\class{\xv} $ in~$ \tau/D $. Then there are $ \vs\in\class{\vr},\vy\in\class{\vx} $ with $ \vs\le\yv $.
 
As neither $ \vr\sim_D\vx $ nor $ \vr\sim_D\xv $ by assumption Lemma~\ref{dist_list} (iv) and (v) imply $ \vx\not\ge\yv $ and $ \vx\not\le\yv $. Furthermore if $ \vx\le\vy $ then $ \vr\le\vx\le\vy\le\sv $, so we are done.

This leaves the case~$ \vx\ge\vy $. If $ \vr\le\yv $ then $ \{\vr,\vy,\xv\} $ is a three-star as in Lemma~\ref{non-trans2} and~\ref{non-trans3}, which we know is impossible as shown in the proof of Lemma~\ref{dist_transitive} if $ D $ is branch-closed. By Lemma~\ref{dist_list}(ii) $ \vr\ge\vy $ would imply $ \vr\sim_D\vx $, so this is also impossible. Hence the only relation $ r $ and $ y $ can have is $ \vr\le\vy $, and then $ \vr\le\vy\le\sv $ as desired.
\end{proof}

The next step is to find $ \vd_1,\vd_2\in D $ with certain relations to the separations from Lemma~\ref{non-triv1}, which we can later apply the assumption that $ D $ is branch-closed to so as to obtain a contradiction.

\begin{LEM}\label{non-triv2}
Let~$ \tau $ be a tree set and $ D $ a selection. Let $ \vr,\vs,\vx\in\tau $ with $ \vr\le\vx\le\sv $ and $ \vr\sim_D\vs $ but neither $ \vr\sim_D\vx $ nor~$ \vr\sim_D\xv $. Then there are $ \vd_1,\vd_2\in D $ with
\[ \vd_1\lne\rv,\sv,\vx,\qquad\vd_2\lne\rv,\sv,\xv. \]
\end{LEM}

\begin{proof}
By $ \vr\le\sv $ and $ \vr\sim_D\vs $ we have $ \dplus{\vr}=\dplus{\vs}\sube\dminus{\vr} $ and hence~$ \dplus{\vr}=\dplus{\vs}=\emptyset $. Therefore either $ \vd\lne\rv,\sv $ or $ \dv\lne\rv,\sv $ for all~$ \vd\in D $. Thus if $ \vd_1\in D $ distinguishes $ \vr $ and $ \vx $ then $ \vd_1\lne\rv,\vx $ is the only possibility, and if $ \vd_2\in D $ distinguishes $ \vs $ and $ \xv $ then $ \vd_2\lne\sv,\xv $ is the only possibility. The claim now follows from the assumption that $ \vr\sim_D\vs $ but neither $ \vr\sim_D\vx $ nor~$ \vs\sim_D\xv $.
\end{proof}

Finally we combine the above lemmas and use Proposition~\ref{star-push} to prove that $ \tau/D $ has no trivial elements.

\begin{LEM}\label{dist_nontriv}
Let~$ \tau $ be a chain-complete tree set and $ D $ a branch-closed selection. Then $ \tau/D $ contains no trivial element.
\end{LEM}

\begin{proof}
Suppose $ \tau/D $ contains a trivial element. From Lemma~\ref{non-triv1} and~\ref{non-triv2} it follows that there are $ \vr,\vs,\vx\in\tau $ with $ \vr\le\vx\le\sv $ and $ \vr\sim_D\vs $ but neither $ \vr\sim_D\vx $ nor~$ \vr\sim_D\xv $, as well as $ \vd_1,\vd_2\in D $ with
\[ \vd_1\lne\rv,\sv,\vx,\qquad\vd_2\lne\rv,\sv,\xv. \]
Proposition~\ref{star-push} applied to the three-star $ \{\vx,\vd_2,\sv\} $ then yields a branching star $ \sigma $ and some $ \vb\in\sigma $ with~$ \vx\le\vb $. Then $ \vd_1\le\vb\le\dv_2 $ by $ \vd_1\le\vx $ and the star property and hence $ \vb\in D $ as $ D $ is branch-closed. But $ \vr\le\vb\lne\sv $ by the star property, so $ \vb $ distinguishes $ \vr $ and $ \vs $, contradicting~$ \vr\sim_D\vs $.

Therefore $ \tau/D $ cannot contain a trivial element.
\end{proof}

We have assembled all the parts necessary to show that $ \tau/D $ is a finite tree set:

\begin{PROP}\label{prop:fintreeset}
Let~$ \tau $ be a chain-complete tree set and $ D $ a branch-closed selection. Then $ \tau/D $ is a finite tree set.
\end{PROP}

\begin{proof}
As $ D $ is finite there are only finitely many subsets of $ D $ and hence only finitely many equivalence classes of~$ \sim_D $, so $ \tau/D $ is finite. The relation $ \le $ on $ \tau/D $ is reflexive by definition, anti-symmetric by Lemma~\ref{anti-symmetric} and transitive by Lemma~\ref{dist_transitive} and thus a partial order. The involution $ \braces{\class{\vs}}^*=\class{\sv} $ is order-reversing: if $ \class{\vs}\le\class{\vr} $ with $ \vs\le\vr $ then $ \sv\ge\rv $ and thus~$ \class{\sv}\ge\class{\rv} $. Therefore $ \tau/D $ is a separation system. Any two unoriented separations $ \{\class{\vs},\class{\sv}\},\{\class{\vr},\class{\rv}\} $ in $ \tau/D $ have comparable orientations, because their representatives $ s $ and $ r $ are nested. Finally Lemma~\ref{dist_nontriv} shows that $ \tau/D $ has no trivial elements and is thus a finite tree set.
\end{proof}

With this we have accomplished the main goal of this section. In the next two sections we will define a suitable directed set $ \D $ of selections of $ \tau $ and show~$ \tau\cong{\invlim(\tau/D\mid D\in\D)} $. To help with this in the remainder of this section we establish a few independent facts about the behaviour of $ \tau/D $ for later use. We show that the relation of $ \vr $ and $ \vs $ in~$ \tau $ can sometimes be recovered from the relation of $ \class{\vr} $ and $ \class{\vs} $ in $ \tau/D $, and that the equivalence classes of~$ \sim_D $ in~$ \tau $ are chain-complete. The latter will be crucial in the surjectivity proof in the next sections.\\ % Furthermore we study the behaviour of~$ \sim_D $ on infinite stars and find a sufficient condition for $ \tau/D $ to be regular.\\
\\
If $ \tau/D $ is a tree set this implies that $ r $ and $ s $ in $ \tau $ have to have the same relation as $ \class{\vr} $ and $ \class{\vs} $ in $ \tau/D $, at least if those are different classes:

\begin{LEM}\label{lem:orderpres}
Let~$ \tau $ be a tree set, $ \vr,\vs\in\tau $ and $ D $ a selection for which $ \tau/D $ is a tree set. If $ \class{\vr}\lne\class{\vs} $ then $ \vr\lne\vs $.
\end{LEM}

\begin{proof}
Any other relation between $ r $ and $ s $ implies either $ \class{\vr}=\class{\vs} $ or that one of $ \class{\vr} $ and $ \class{\sv} $ would be trivial in~$ \tau/D $.
\end{proof}

For the study of $ \tau/D $ it is essential to know the behaviour of chains of~$ \tau $ with regard to~$ \sim_D $. It turns out that the equivalence classes of~$ \tau $ are chain-complete themselves if~$ \tau $ is; we don't even need the assumption that $ D $ is branch-closed for this:

\begin{PROP}\label{prop:classeschainclosed}
Let~$ \tau $ be a chain-complete tree set, $ D $ a selection, and $ \vt\in\tau $. Then $ \class{\vt} $ is chain-complete.

In particular $ \class{\vt} $ has a maximal (and a minimal) element.
\end{PROP}

\begin{proof}
Let $ C $ be a chain in the equivalence class $ \class{\vt} $ with supremum $ \vs $ in~$ \tau $. The claim is trivial if $ \vs\in C $. Thus we may assume that $ \vs\notin C $. Then $ \sv $ cannot lie in a splitting star $ \sigma $ of~$ \tau $, as in that case some other element $ \vsdash $ of $ \sigma $ would be an upper bound of $ C $ with $ \vsdash\lne\vs $.

Pick some $ \vr\in C $; we will verify that $ \vr\sim_D\vs $. As $ \vr\le\vs $ we have $ \dplus{\vr}\sube\dplus{\vs} $ and $ \dminus{\vs}\sube\dminus{\vr} $.

Consider $ \vd\in\dminus{\vr} $. As all elements of $ C $ are $ D $-equivalent $ \dv $ is an upper bound for $ C $ and hence $ \vs\le\dv $. On the one hand $ \vd\ne\sv $ as $ \sv $ does not lie in a splitting star, on the other hand $ \vd\ne\vs $ as then $ \vr $ would be trivial. Therefore $ \vs\lne\dv $ and thus $ \dminus{\vr}\sube\dminus{\vs} $.

Now consider $ \vd\in\dplus{\vs} $. This $ \vd $ cannot be an upper bound for $ C $, hence either $ \vd\in\dplus{\vr} $ or $ \vd\le\rv $. In the latter case $ \dv $ is an upper bound for $ C $ implying $ \vd\lne\vs\le\dv $ and thus that $ \vd $ would be trivial. Therefore $ \dplus{\vs}\sube\dplus{\vr} $ and hence $ \vr\sim_D\vs $.
\end{proof}

For a subset $ B\sub\tau $ and a selection $ D $ write $ \class{B}:=\{\class{\vb}\mid \vb\in B\}\sub\tau/D $. A direct consequence of Lemma~\ref{dist_list}(ii) and Proposition~\ref{prop:classeschainclosed} is that for a chain $ C\sube\tau $ the supremum of $ \class{C} $ is the class of the supremum of $ C $ in~$ \tau $:

\begin{COR}\label{cor:chainsup}
Let~$ \tau $ be a chain-complete tree set, $ D $ a branch-closed selection,~$ C $ a chain and $ \vs $ the supremum of $ C $ in~$ \tau $. Then $ \class{\vs}=\max\,\,\class{C} $ in~$ \tau/D $.
\end{COR}

\begin{proof}
The relation $ \sim_D $ has finitely many equivalence classes, so Lemma~\ref{dist_list}(ii) implies that some final segment of $ C $ is completely contained in some class $ \class{\vt} $ of~$ \tau $. The set $ \class{C} $ in $ \tau/D $ is again a chain, and as $ \class{\vt} $ contains a final segment of $ C $ it is the maximum of $ \class{C} $ in~$ \tau/D $. Proposition~\ref{prop:classeschainclosed} now implies~$ \vs\in\class{\vt}=\max\class{C} $.
\end{proof}

Infinite splitting stars of $ \tau $ play an important role in the upcoming Section~\ref{sec:proof}. We now analyze their behaviour with regard to~$ \sim_D $. This turns out to be quite simple: if a splitting star $ \sigma $ meets $ D $, then all elements of $ \sigma\cap D $ are pairwise non-equivalent, and all elements of $ \sigma\setminus D $ get identified:

\begin{LEM}\label{lem:infiniteclass}
Let~$ \tau $ be a chain-complete tree set, $ D $ a branch-closed selection and $ \sigma $ a splitting star that meets~$ D $. Then $ \vr\sim_D\vs $ for distinct $ \vr,\vs\in\sigma $ if and only if~$ \vr,\vs\notin D  $.

In particular if $ \sigma $ is infinite there is exactly one equivalence class of~$ \sim_D $ containing infinitely many elements of $ \sigma $, and every other equivalence class contains at most one element of~$ \sigma $.
\end{LEM}

\begin{proof}
Let $ \vr,\vs\in\sigma $ be two distinct separations. For the forward direction suppose that $ \vr\in D  $. Then $ \vr\in\dminus{\vs} $ but $ \vr\notin\dminus{\vr} $, so $ \vr\not\sim_D\vs $.

For the backward direction assume that $ \vr,\vs\notin D  $. Then $ \dplus{\vs}=\emptyset $ as otherwise $ \vs\in D  $ by the assumptions that $ D $ is branch-closed and $ \sigma $ meets~$ D $. Similarly $ \dplus{\vr}=\emptyset $. Moreover $ \dminus{\vr}\setminus\{\vs\}=\dminus{\vs}\cup\dplus{\vs} $ as $ \vr,\vs $ lie in a splitting star, so $ \dminus{\vr}=\dminus{\vs} $ by $ \vs\notin D  $.
\end{proof}

To apply Lemma~\ref{lem:infiniteclass} in practice it is useful to have a sufficient condition for $ \sigma $ to meet~$ D $. The following lemma accomplishes this by showing that a splitting star~$ \sigma $ of~$ \tau $ must meet $ D $ as soon as it meets at least three equivalence classes of~$ \sim_D $:

\begin{LEM}\label{lem:threeclasses}
Let~$ \tau $ be a chain-complete tree set, $ D $ a branch-closed selection, and $ \sigma $ a splitting star which meets at least three equivalence classes of~$ \sim_D $. Then $ \sigma $ meets~$ D $.
\end{LEM}

\begin{proof}
Suppose that $ \sigma\cap D =\emptyset $. Then there is $ \vt\in\sigma $ with $ \dplus{\vt}\ne\emptyset $. Consider $ \vr,\vs\in\sigma $ with $ \vr\not\sim_D\vt $ and~$ \vs\not\sim_D\vt $. We will show that $ \vr\sim_D\vs $, contradicting the assumption that $ \sigma $ meets three equivalence classes. First note that $ \dplus{\vr}=\dplus{\vs}=\emptyset $ by the assumptions that $ D $ is branch-closed and~$ \vt\notin~D $.  Moreover $ \dminus{\vr}\setminus\{\vs\}=\dminus{\vs}\cup\dplus{\vs} $ as $ \vr,\vs $ lie in a splitting star, so $ \dminus{\vr}=\dminus{\vs} $ by $ \vs\notin D $. Hence $ \vr\sim_D\vs $.
\end{proof}

\subsection{Obtaining profinite tree sets}\label{sec:proof}

In this section we prove the second assertion of Theorem~\ref{thm:treesets}:

\ThmTreesets*

For this we find a set of properties such that we can obtain every tree set with these properties as an inverse limit of finite tree sets, making use of the `quotients' defined in the previous section. Following that we show that every profinite tree set has these properties. In doing this we also obtain a characterization of the profinite tree sets in purely combinatorial terms. 

The general strategy is as follows. For a tree set $ \tau $ we define a suitable directed set $ \D $ of selections such that $ \tau/D $ is a finite tree set for each $ D\in\D $. We will use Proposition~\ref{prop:fintreeset} for this, so $ \tau $ needs to be chain-complete and the selections in $ \D $ must be branch-closed. For $ \vs\in\tau $ and selections $ D\sube D' $ we then have $ [\vs]_{D'}\sube\class{\vs} $, so by taking these inclusions as the bonding maps $ {(\tau/D\mid D\in\D)} $ is an inverse system of finite tree sets. It then remains to prove that $ \tau $ and $ \tau':={\invlim(\tau/D\mid D\in\D)} $ are isomorphic; for this we define the map $ \ph\colon\tau\to\tau' $ as
\[ \ph(\vs)=(\class{\vs}\mid D\in\D). \]
By the observation above $ \ph(\vs) $ is indeed always a compatible choice (and thus an element of $ \tau' $), and $ \ph $ is a homomorphism of separation systems by Lemma~\ref{dist_list}(i) and the definition of~$ \le $ in~$ \tau/D $. We shall verify the assumptions of Lemma~\ref{lem:isononreg}: that $ \ph $ is a bijection and that pre-images of small separations are small. The latter will be done by simple case-checking.

The map $ \ph $ is injective if and only if for all distinct $ \vr,\vs\in\tau $ the set $ \D $ contains a selection $ D $ with~$ \class{\vr}\ne\class{\vs} $. An almost sufficient condition for this is that there is a splitting star of $ \tau $ between any two given separations. To make this formal we say that a tree set $ \tau $ is {\em splittable} if for every $ \vr,\vs\in\tau $ with $ \vr\lne\vs $ there is a splitting star $ \sigma $ of~$ \tau $ with $ \vrdash,\svdash\in\sigma $ such that~$ \vr\le\vrdash\lne\vsdash\le\vs $. Then $ \ph $ is injective if $ \tau $ is splittable \textit{and} the set $ \bigcup\D $ contains all elements of $ \tau $ that lie in non-singleton splitting stars\footnote{In fact we will use a slightly smaller but still sufficient set $ \D $ of selections: $ \bigcup\D $ will miss only a single separation of every infinite splitting star, which we will deal with separately.}.

For the surjectivity of $ \ph $ we will rely on Corollary~\ref{cor:chainsup}. This is the most technical part of the proof. Essentially, if $ \tau' $ contains some separation $ \vtstar $ which is not in $ \ph(\tau) $, we will try to `sandwich' $ \vtstar $ by taking maximal chains $ C $ and $ C' $ in $ \tau $ whose images under $ \ph $ lie below and above $ \vtstar $ respectively. Then $ \vtstar $ will lie between the images under $ \ph $ of the supremum $ \vs $ of $ C $ and the infimum $ \vsdash $ of $ C' $. However, no element of $ \tau $ will lie inbetween $ \vs $ and $ \vsdash $, enabling us to show that for sufficiently large selections $ D\in\D $ there can be no compatible choice of elements of $ \tau/D $ inbetween $ \class{\vs} $ and $ \class{\vsdash} $ apart from $ \ph(\vs) $ and $ \ph(\vsdash) $ themselves. This approach works as long as $ \vtstar $ is not maximal or minimal in $ \tau' $; in those cases we need to be more careful, and we will adress this below.

What remains is the choice of the selections in $ \D $. Proposition~\ref{prop:fintreeset} demands that all $ D\in\D $ are branch-closed. For the injectivity of $ \ph $ we want $ \bigcup\D $ to contain all non-singleton splitting stars of $ \tau $. Lastly $ \D $ needs to be a directed set. However if $ \tau $ contains two non-singleton splitting stars with infinitely many branching points inbetween them then clearly we cannot achieve all three of these simultaneously. We therefore need to assume that $ \tau $ contains only finitely many branching points inbetween any two non-singleton splitting stars.

To make this formal, for unoriented separations $ s,s'\in\tau $ let $ C(s,s') $ denote the set of all branching points $ \vb $ of $ \tau $ for which there are orientations $ \vs,\vsdash $ with $ \vs\le\vd\le\vsdash $ or $ \vsdash\le\vd\le\vs $. Then $ C(s,s') $ is always the disjoint union of two chains, and if $ C(s,s') $ meets a splitting star $ \sigma $ in an element other than $ s $ or $ s' $ then it meets $ \sigma $ in exactly two elements. Thus for $ \vs\le\vsdash $, if there is a $ \vd\in C(s,s') $ with $ \vs\le\vd\lne\vsdash $, then $ C(s,s') $ contains a $ \dvdash $ with $ \vs\le\vd\lne\vddash\le\vsdash $. Furthermore a selection $ D $ of $ \tau $ is branch-closed if and only if $ C(s,s')\sube D $ for all~$ \vs,\vsdash\in D $.

If we assume that $ C(s,s') $ is finite for all regular $ s,s' $ in $ \tau $ then no two non-singleton splitting stars can have infinitely many branching points between them as each of these stars must contain a regular separation. Under this assumption the set of all branch-closed selections is a directed set.

However the set of all branch-closed selections is a bit too large to ensure the surjectivity of $ \ph $; consider for example an infinite splitting star $ \sigma $ of $ \tau $. As seen in Lemma~\ref{lem:infiniteclass} for every branch-closed selection $ D $ that meets $ \sigma $ the set $ \sigma\setminus D $ lies inside one equivalence class of $ \sim_D $. The family $ \vtstar:= (\class{\sigma\setminus D}\mid D\cap\sigma\ne\emptyset) $ is a compatible choice which unfortunately is not in the image of $ \ph $: for any $ \vs\in\sigma $ there is some selection $ D $ that contains $ \vs $ and hence shows that $ \ph(\vs)\ne\vtstar $. Therefore we need to `reserve' some $ \vs\in\sigma $ as the $ \vs\in\tau $ with $ \ph(\vs)=\vtdash $, which we can achieve by putting only those branch-closed selections $ D $ of $ \tau $ into $ \D $ that do not contain $ \vs $. Since $ \vtstar $ is small we need to pick a small separation $ \vs\in\sigma $ for this. The assumption on $ \tau $ needed for this to work is thus that every infinite splitting star contains a small separation. The fact that $ \bigcup\D $ now misses this selected $ \vs $ does not interfere with the injectivity of $ \ph $: since $ \vs $ is small and hence minimal in $ \tau $ it is not needed to distinguish any two separations of $ \tau $ and is also never required to be in a selection to make it branch-closed.

We will now show that under the assumptions outlined above we do indeed get an isomorphism between $ \tau $ and $ \tau' $.

\begin{PROP}\label{prop:nonregisom}
Let $ \tau $ be a chain-complete splittable tree set with no infinite regular splitting star, in which $ C(s,s') $ is finite for all regular~$ s,s' $ in~$ \tau $. For every infinite splitting star $ \sigma $ let $ \nu(\sigma) $ be a small separation in~$ \sigma $, and let $ \D $ be the set of all branch-closed selections $ D $ of $ \tau $ with $ \nu(\sigma)\notin D $ for every infinite splitting star $ \sigma $ of~$ \tau $.

Then $ {(\tau/D\mid D\in\D)} $ is an inverse system of finite tree sets and the map~$ \ph\colon\tau\to{\invlim(\tau/D\mid D\in\D)} $ with
\[ \ph(\vs)=(\class{\vs}\mid D\in\D) \]
is an isomorphism of tree sets.
\end{PROP}

\begin{proof}
By Proposition~\ref{prop:fintreeset} each $ \tau/D $ is a finite tree set. The set $ \D $ ordered by inclusion is a directed set: to see this, let $ D,D'\in\D $ and set
\[ E:=D \cup D'\cup\bigcup_{\vs,\vsdash\in D\cup D'}C(s,s'). \]
Then $ E\in\D $ with $ D,D'\sube E $: $ E $ is finite by the assumption that $ C(s,s') $ is finite for all regular~$ s,s' $, and $ E $ is branch-closed by construction. It contains no $ \nu(\sigma) $ of any infinite splitting star $ \sigma $, as neither $ D $ nor $ D' $ does and thus $ \nu(\sigma)\notin C(s,s') $ for all $ \vs,\vsdash\in D\cup D' $. Furthermore $ E $ meets no splitting star in exactly one element: since $ D $ and $ D' $ are selections, neither of them meets any splitting star in exactly one separation; and for $ C(s,s') $ with $ \vs,\vsdash\in D\cup D' $, if $ C(s,s') $ contains a element of some splitting star, and that star does not already meet $ D $ or $ D' $, then $ C(s,s') $ contains at least two elements of that star. Thus $ E $ is a selection with $ E\in\D $, showing that $ \D $ is indeed a directed set.

Therefore $ {(\tau/D\mid D\in\D)} $ is an inverse system with the surjective bonding maps $ f_{DD'}\colon\tau/D\to\tau/D' $ defined as
\[ f(\class{\vs})=[\vs]_{D'} \]
for~$ D'\subset D $. Note that these bonding maps are well-defined by the definition of~$ \sim_D $, and are homomorphisms of tree sets by Lemma~\ref{dist_list}(i). Thus $ \tau':={\invlim(\tau/D\mid D\in\D)} $ is a tree set by Theorem~\ref{thm:treesets}(i).

The map $ \ph $ is a homomorphism of tree sets by Lemma~\ref{dist_list}(i) and the definition of $ \le $ in~$ \tau/D $: if $ \vr\le\vs $ then $ \class{\vr}\le\class{\vs} $ for all $ D\in\D $ and hence~$ \ph(\vr)\le\ph(\vs) $. The claim thus follows from Lemma~\ref{lem:isononreg} if we can show that $ \ph $ is a bijection, and that pre-images of small separations are small.

For the injectivity let $ \vr,\vs\in\tau $ be two distinct separations. Then $ \ph(\vr)\ne\ph(\vs) $ follows from the assumption that $ \tau $ is splittable, unless one of $ \vr $ and $ \vs $ is $ \nu(\sigma) $ for some infinite splitting star $ \sigma $. Suppose $ \vr=\nu(\sigma) $. If $ s $ does not meet $ \sigma $ the injectivity again follows from $ \tau $ being splittable; if on the other hand $ s $ meets $ \sigma $ then any $ D\in\D $ that meets $ s $ witnesses $ \ph(\vr)\ne\ph(\vs) $ by Proposition~\ref{lem:infiniteclass}.

To show that pre-images of small separations are small we will show that the images of non-small separations are non-small. So consider a non-small $ \vs\in\tau $; we will find a $ D\in\D $ for which $ \class{\vs} $ is non-small, witnessing that $ \ph(\vs) $ is not small. If there are $ \vr,\vt\in\tau $ with $ \vr\lne\vs\lne\vt $, then by Lemma~\ref{lem:orderpres} we can obtain a suitable selection $ D\in\D $ by applying the splittability of~$ \tau $ to the pair $ \vr,\vs $ and $ \vs,\vt $, and then taking as $ D $ those two-stars and all branching points between them. We may therefore assume that there is no $ \vr\in\tau $ with $ \vr\lne\vs $ (the other case is symmetrical). If $ \vs $ lies in a splitting star $ \sigma $ of~$ \tau $, then $ \sigma $ is not a singleton star, and any non-singleton finite subset of $ \sigma $ containing $ \vs $ but not $ \nu(\sigma) $ is a selection $ D\in\D $ for which $ \class{\vs} $ is regular. So suppose that $ \vs $ does not lie in a splitting star. Then $ \vs $ cannot be co-small, and there is a separation $ \vt\in\tau $ with $ \vs\lne\vt $ and $ t $ regular. Consider $ C(s,t) $, which is finite by assumption since $ s $ and $ t $ are both regular. If $ C(s,t) $ is empty then by applying the splittability of~$ \tau $  to $ \vs,\vt $ we obtain a two-element subset of a splitting star which is a selection $ D\in\D $ with $ \class{\vs} $ regular. So suppose that $ C(s,t) $ is non-empty. From the assumption that there is no $ \vr\in\tau $ with $ \vr\lne\vs $ we know that $ \sv $ cannot be a branching point of $ \tau $ and hence $ \sv\notin C(s,t) $. Moreover $ \vs\notin C(s,t) $ since we assumed that $ \vs $ does not lie in any splitting star of $ \tau $. Thus there exists a $ \vd\in C(s,t) $ with $ \vs\lne\vd $; pick a minimal such $ \vd $. Let $ \sigma $ be the branching star containing $ \vd $, and let $ \dvdash $ be any element of $ \sigma $ other than $ \vd $ or $ \nu(\sigma) $ (if $ \sigma $ is infinite). Note that, if $ \sigma $ is infinite, then $ \vd $ cannot be $ \nu(\sigma) $ since then $ \vs\lne\vd $ would be small. Thus $ D:=\{\vd,\dvdash\} $ is a selection in $ \D $, and the minimality of $ \vd $ implies that $ \class{\vs} $ cannot be small.

For the surjectivity of $ \ph $ let $ \vtstar=(\vtstarD\mid D\in\D)\in\tau'={\invlim(\tau/D\mid D\in\D)} $ and assume for a contradiction that there is no $ \vs\in\tau $ with $ \ph(\vs)=\vtstar $. This implies that $ \vtstarD $ is an infinite equivalence class for each $ D\in\D $: for if $ \vtstarD $ were finite for some $ D\in\D $, then the injectivity of $ \ph $ and the fact that $ \D $ is a directed set would enable us to find a $ D'\subseteq D $ in $ \D $ for which $ \vtstar_{D'} $ is a singleton, in which case $ \vtstar $ would be the image of the single element of~$ \vtstar_{D'} $.

 Set
\[ X:=\menge{\vs\in\tau\mid \ph(\vs)\le\vtstar},\qquad Y:=\menge{\vs\in\tau\mid \ph(\vs)\ge\vtstar}. \]
For every $ s\in\tau $ exactly one of $ \vs,\sv $ lies in $ X\cup Y $. Therefore at most one of $ X $ and $ Y $ is empty. We distinguish two cases.

\textbf{Case 1:} Both $ X $ and $ Y $ are non-empty.

Let $ C $ be a maximal chain in $ X $ with supremum $ \vs $ in $ \tau $ and $ C' $ a maximal chain in $ Y $ with infimum $ \vsdash $. Then $ \ph(C) $ and $ \ph(C') $ are chains too. From Corollary~\ref{cor:chainsup} it follows that $ \ph(\vs) $ and $ \ph(\vsdash) $ are respectively the supremum of~$ \ph(C) $ and the infimum of~$ \ph(C') $ in $ \tau' $, hence $ \ph(\vs)\le\vtstar\le\ph(\vsdash) $. In fact we even have $ \ph(\vs)\lne\vtstar\lne\ph(\vsdash) $ by the assumption that $ \vtdash $ does not lie in the image of $ \ph $. Pick $ D\in\D $ large enough that $ \class{\vs}\le\vtstarD\le\class{\vsdash} $ with $ \vtstarD\ne\class{\vs} $ and $ \vtstarD\ne\class{\vsdash} $. As $ \tau/D $ is a tree set this implies $ \class{\vs}\lne\vtstarD\lne\class{\vsdash} $. Now consider some $ \vt\in\tau $ with $ \class{\vt}=\vtstarD $. By Lemma~\ref{lem:orderpres} we have $ \vs\lne\vt\lne\vsdash $, so $ \tv\in X\cup Y $ would imply that one of $ \vs $ and $ \vs' $ is co-trivial. Therefore $ \vt\in X\cup Y $, which contradicts the maximality of either $ C $ or $ C' $. This concludes Case~1.

\textbf{Case 2:} One of $ X $ and $ Y $ is empty.

We may assume that $ Y $ is empty (the case that $ X $ is empty is analogous).

Let $ C $ be a maximal chain in $ X $ with supremum $ \vs $. By Corollary~\ref{cor:chainsup} $ \ph(\vs) $ is the supremum of $ C $ in $ \tau' $ and hence $ \ph(\vs)\le\vtstar $. Moreover $ \ph(\sv)\ne\vtstar $ and therefore $ \ph(\vs)\lne\vtstar $. Let $ \D' $ be the set of all $ D\in\D $ with $ \class{\vs}\lne\vtstarD $. This is a cofinal set in $ \D $. For each $ D\in\D' $ let $ \vM(D) $ be the set of minimal elements of the equivalence class $ \vtstarD $, which is non-empty by Proposition~\ref{prop:classeschainclosed}. Consider a $ D\in\D' $ and~$ \vr\in\vM(D) $. Then $ \vs\lne\vr $, and by the maximality of $ C $ and the minimality of $ \vr $ there can be no $ \vt\in\tau $ with~$ \vs\lne\vt\lne\vr $. Lemma~\ref{lem:twoclosedgeneral} thus implies that $ \vs $ and $ \rv $ lie in a common splitting star $ \sigma $ of~$ \tau $. As $ \sigma $ is the unique splitting star of~$ \tau $ containing $ \vs $, and both $ D\in\D' $ and $ \vr\in\vM(D) $ were arbitrary, this shows $ \Mv(D):=\{\rv\mid \vr\in\vM(D)\}\sube\sigma $ for every~$ D\in\D' $. We will show that $ \vM(D) $ and thus $ \sigma $ is infinite for every $ D\in\D' $ and deduce~$ \tvstar=\ph(\nu(\sigma)) $.

Suppose for a contradiction that there is a $ D\in\D' $ with $ \vM(D) $ finite and let $ D'\in\D' $ with $ D'\supseteq D $ be such that $ [\vr]_{D'}\ne\vtstar_{D'} $ for each $ \vr\in\vM $. Pick a~$ \vu\in\vM(D') $. Then $ \uv\in\Mv(D')\sub\sigma $. By compatibility $ \class{\vu}=\vtstarD $ and hence $ \vr\le\vu $ with $ \vr\ne\vu $ for some $ \vr\in\vM(D) $, contradicting $ \rv,\uv\in\sigma $. Therefore $ \vM(D)\sub\sigma $ is infinite for every $ D\in\D' $.

Pick a $ \rv\in\sigma $ with $ \rv\ne\vs $ and fix some $ D\in\D' $ with $ \class{\rv}\ne\class{\vs} $ and $ \class{\rv}\ne\vtstarD $. Then Lemma~\ref{lem:infiniteclass} and Lemma~\ref{lem:threeclasses} imply that $ \tvstar_{D'}=[\nu(\sigma)]_{D'} $ for every $ D'\in\D' $ with $ D'\ge D $ as $ \nu(\sigma)\notin D' $ by the definition of $ \D $. As the set of all $ D'\in\D' $ with $ D'\ge D $ is cofinal in $ \D $ this shows~$ \tvstar=\ph(\nu(\sigma)) $. This concludes Case~2.

We have shown that $ \ph $ is a bijection. The claim now follows from Lemma~\ref{lem:isononreg}.
\end{proof}

In order to establish Theorem~\ref{thm:treesets} we now show that all profinite tree sets meet the assumptions of Proposition~\ref{prop:nonregisom}, i.e. that every profinite tree set is chain-complete, splittable, contains no infinite regular splitting star, and has finite $ C(s,s') $ for all regular separations $ s $ and $ s' $.

The following lemma from~\cite{ProfiniteASS} implies that every profinite separation system is chain-complete:

\begin{LEM}[\cite{ProfiniteASS}]\label{lem:chains}
Let $ \S=(\vSp\mid p\in P) $ be an inverse system of finite separation systems and $ \vS=\invlim\S $. If $ C\sub\vS $ is a non-empty chain, then $ C $ has a supremum and an infimum in~$ \vS $. Both these lie in the closure of $ C $ in~$ \vS $.\qed
\end{LEM}

Let us now show that profinite tree sets are splittable, using the knowledge that they are chain-complete:

\begin{PROP}\label{prop:splittable}
	Profinite tree sets are splittable.
\end{PROP}

\begin{proof}
	Let $\tau=\invlim (S_p\mid p\in P)$ be a profinite tree set and $ \vr,\vs\in\tau $ with $ \vr\lne\vs $. Fix $ p\in P $ such that $ \vr_p\lne\vs_p $ and set
	\[ X:=\menge{\vx\in\tau\mid \vr\le\vx\le\vs\tn{ and }\vx_p=\vr_p}. \]
	Then $ X $ is a chain with $ \vr\in X $ which by Lemma~\ref{lem:chains} has a supremum $ \vrdash=(\max X_q\mid q\in P) $. Now set
	\[ Y:=\menge{\vy\in\tau\mid \vr\le\vy\le\vs\tn{ with }\vr_p\le\vy_p\tn{ and }\vr_p\ne\vy_p}. \]
	This is a chain with $ \vs\in Y $ and infimum $ \vsdash=(\min Y_q\mid q\in P) $. By definition we have $ \vr\le\vrdash\lne\vsdash\le\vs $. Furthermore no $ \vt\in\tau $ lies strictly between $ \vrdash $ and $\vsdash $ as such a $ \vt $ would lie in $ X\cup Y $ and thus contradict the definition of either $ \vrdash $ or~$ \vsdash $. By Lemma~\ref{lem:extension} there is a consistent orientation~$ O $ of~$ \tau $ extending $ \{\vrdash,\svdash\} $ in which $ \vrdash $ and therefore $ \svdash $ is maximal. Lemma~\ref{lem:twoclosedgeneral} says that $ O $ is splitting, so $ \vrdash $ and $ \svdash $ lie in a common splitting star.
\end{proof}

The next assumption made by Proposition~\ref{prop:nonregisom} is that every infinite splitting star of the tree set at hand contains a small separation. This is the same as asking that the tree set contains no regular infinite splitting star. In fact we can show slightly more for profinite tree sets:

\begin{PROP}\label{prop:starfinite}
	Let $\tau$ be a profinite tree set. Then every infinite star which is maximal by inclusion contains a small separation. 
\end{PROP}

\begin{proof}
	Suppose $\sigma\subseteq\tau$ is an infinite maximal star, and~$\tau=\invlim (S_p\mid p\in P)$. Let $\sigma_p$ be the projection of $\sigma$ to~$S_p$. For every $p\in P$ there must be some $\vs_p\in S_p$ which is the image of infinitely many elements of~$\sigma$. As $\sigma$ is a star such a $\vs_p$ has to be small. For $p\in P$ let $\sigma_p'$ be the set of all $\vs_p\in S_p$ which are the projection of infinitely many elements of~$\sigma$. Then $(\sigma_p'\mid p\in P)$ is an inverse system of finite sets with a non-empty inverse limit, and its elements are also elements of~$\tau$. Let $\vs\in\invlim (\sigma_p'\mid p\in P)$ be such an element. As every $\vs_p$ is small so is~$\vs$. Moreover $\vs_p\le\rv_p$ for all $p\in P$ and $\vr\in\sigma$, so $\vs\in\sigma$ by maximality.
\end{proof}

As every splitting star of a tree set is also an inclusion-maximal star, Proposition~\ref{prop:starfinite} clearly implies that profinite tree sets contain no regular infinite splitting stars.

Let us call a tree set~{\em star-finite} if it contains no infinite star. Then Proposition~\ref{prop:starfinite} implies that all regular profinite tree sets are star-finite:

\begin{COR}
	Regular profinite tree sets are star-finite.
\end{COR}

\begin{proof}
	If a profinite tree set contains an infinite star by Proposition~\ref{prop:starfinite} it also contains a small separation. Hence regular profinite tree sets do not contain infinite stars.
\end{proof}

A tree set $ \tau $ that contains no infinite star clearly contains no regular infinite splitting star either. Furthermore if $ C(s,s') $ was infinite for any $ s,s'\in\tau $ then the set of all separations in $ \tau\setminus C(s,s') $ that belong to a branching star of $ \tau $ which meets $ C(s,s') $ is an infinite star in $ \tau $. Therefore~Lemma~\ref{lem:chains} together with Proposition~\ref{prop:nonregisom},~\ref{prop:splittable} and~\ref{prop:starfinite} implies the following characterization of the regular profinite tree sets:

\begin{THM}\label{thm:profinitecharregular}
	{\em A regular tree set $\tau$ is profinite if and only if it is chain-complete, splittable and star-finite.}\hfill$\Box$
\end{THM}

To complete the proof of Theorem~\ref{thm:treesets} it remains to show that that $ C(s,s') $ is finite for all regular $ s,s' $ in a profinite tree set $ \tau $. We do this in three steps. First we show that every infinite chain has some limit element. Then we show that if $ \vm\in\tau $ is the supremum of a chain of branching points it must be co-small; and finally we infer that $ C(s,s') $ can only be infinite if one of~$ s $ and~$ s' $ is non-regular.

The first step is more about posets and chain-complete tree sets than about profinite tree sets:

\begin{LEM}\label{lem:middlelimit}
Let~$ \tau $ be a chain-complete tree set and $ C $ an infinite chain in~$ \tau $. Then there is a sub-chain $ C'\sube C $ that does not contain both its infimum and its supremum in~$ \tau $.
\end{LEM}

\begin{proof}
We may assume that $ C $ contains its infimum and supremum in~$ \tau $ as otherwise $ C':=C $ is as desired.

Let us define
\[ C_{\le\vs}:=\menge{\vr\in C\mid \vr\le\vs} \]
for $ \vs\in C $ and define
\[ L:=\menge{\vs\in C\mid C_{\le\vs}\tn{ is finite}}. \]
Then $ L $ is a non-empty sub-chain of $ C $. Let $ \vl $ be the supremum of $ L $ in~$ \tau $; if $ \vl\notin L $ then $ L $ is as desired. If on the other hand $ \vl\in L $ then $ L $ is finite, so $ R:=C\setminus L $ is infinite. Let $ \vr $ be the infimum of $ R $ in~$ \tau $. If $ \vr\in R $ then $ C_{\le\vr} $ is infinite, so there is a $ \vt\in C_{\le\vr}\setminus (L\cup\{\vr\}) $. But this contradicts the fact that $ \vr $ is the infimum of $ R $. Therefore $ \vr\notin R $ and $ R $ is the desired sub-chain.
\end{proof}

Now we prove that the supremum of a chain of branching points is co-small. The proof of this is somewhat analogous to the proof of Proposition~\ref{prop:starfinite}:

\begin{LEM}\label{lem:branchingchain}
Let $ \tau=\invlim(S_p\mid p\in P) $ be a profinite tree set, $ C $ a chain of infinitely many branching points and $ \vm $ the supremum of $ C $ in~$ \tau $. If~$ \vm\notin C $ then~$ \vm $ is co-small.
\end{LEM}

\begin{proof}
As co-small separations are maximal in~$ \tau $ we may assume without loss of generality that $ C $ is a chain of order type~$ \omega $.

Let $ C=\{\vs^n\mid n\in\N\} $ with $ \vs^n\lne\vs^{n+1} $ for all~$ n\in\N $. For every $ n\in\N $ pick an element $ \vt^n $ of the branching star containing $ \vs^n $ with~$ \vt^n\lne\vs^{n+1} $. Let $ \vm=(\vm_p\mid p\in P) $ be the supremum of $ C $. Then by Lemma~\ref{lem:chains} for any fixed $ p\in P $ there is an $ n\in\N $ with $ \vs^n_p=\vm_p $; let $ k(p)\in\N $ be the minimal such index and write
\[ T_p:=\menge{\vt^n_p\mid n\ge k(p)}. \]
Observe that if $ n\ge k(p) $ for some $ p\in P $ then $ \vt^n\le\vm,\sv^n $ by definition and hence $ \vt^n_p\le\vm_p $ as well as $\vt^n_p\le\sv^n_p=\mv_p $, and as a consequence also~$ \vt^n_p\le\tv^n_p $.

Moreover $ k(p)\le k(q) $ for all $ p,q\in P $ with $ p\le q $, so $ (T_p\mid p\in P) $ is an inverse system whose inverse limit is a subset of~$ \tau $. Let~$ \vt\in\invlim(T_p\mid p\in P) $. By the above observation we have $ \vt\le\vm $ as well as $ \vt\le\mv $, and thus $ t=m $, since otherwise $ m $ would witness that $ \vt $ is trivial. Moreover $ \vt $ is small by the above observation. Therefore one of $ \vm $ and $ \mv $ is small; but the first of these is impossible since then then every $ \vs^n $ would be trvial. Therefore $ \mv $ is small, that is, $ \vm $ is co-small.
\end{proof}

With these two lemmas we can now show that $ C(s,s') $ is finite for all regular~$ s,s' $ in profinite tree sets:

\begin{PROP}\label{lem:C(s,s')}
Let~$ \tau $ be a profinite tree set and $ s,s'\in\tau $ two regular unoriented separations. Then $ C(s,s') $ is finite.
\end{PROP}

\begin{proof}
Suppose that $ C(s,s') $ is infinite. Then $ C(s,s') $ is the disjoint union of two infinite chains. Let $ C $ be one of them. By Lemma~\ref{lem:middlelimit} there is a sub-chain $ C' $ of $ C $ that does not contain both its infimum and supremum in~$ \tau $; suppose that $ C' $ does not contain its supremum (the other case is symmetrical). Let $ \vm $ be the supremum of $ C' $. Lemma~\ref{lem:branchingchain} implies that $ \vm $ is co-small. But one of $ \vs,\sv,\vsdash $ or $ \svdash $ is an upper bound for $ C' $. As co-small elements are maximal in~$ \tau $ it follows that $ m=s $ or $ m=s' $, a contradiction.
\end{proof}

Lemma~\ref{lem:chains} and Proposition~\ref{prop:nonregisom}, \ref{prop:splittable}, \ref{prop:starfinite} and~\ref{lem:C(s,s')} combine into the following theorem characterizing the profinite tree sets:

\begin{THM}\label{thm:profinitechar}
{\em A tree set~$ \tau $ is profinite if and only if it is chain-complete and splittable, contains no regular infinite splitting star, and has the property that~$ C(s,s') $ is finite for all regular~$ s,s' $.}\hfill$ \Box $
\end{THM}

Moreover we can now prove Theorem~\ref{thm:treesets}(ii), that is, that every profinite tree set is an inverse limit of finite tree sets:\\
\\
{\em Proof of Theorem~\ref{thm:treesets}}(ii). Let $ \tau $ be a profinite tree set. From Theorem~\ref{thm:profinitechar} it follows that~$ \tau $ meets the assumptions of Proposition~\ref{prop:nonregisom}, which together with Proposition~\ref{prop:fintreeset} implies that $ \tau $ is an inverse limit of finite tree sets.\hfill$ \Box $\\
\\
Therefore the profinite tree sets are indeed precisely those tree sets that are an inverse limit of finite tree sets.\\
\\
Let us conclude this section by proving by example that the four properties in Theorem~\ref{thm:profinitechar} are independent of each other, in the sense that none of them is implied by the others:

\begin{EX}\label{ex:chain-complete}
	Let $ R $ be a one-way infinite ray. Then the edge tree set $ \tau(R) $ is a regular tree set which is splittable and star-finite but not chain-complete.\\
\end{EX}

\begin{EX}\label{ex:splittable}
	Let
	\[ \tau:=\braces{[-2,-1]\cup[1,2],\,\curlyle\,,\,*}, \]
	where $ x^*:=-x $, and $ x\curlyle y $ if and only if $ x\le y $ as real numbers and $ x $ and $ y $ have the same sign. Then~$ \tau $ is chain-complete and star-finite, but $ \tau $ is not splittable since the only splitting stars are~$ \{-1\} $ and~$ \{2\} $.\\
\end{EX}

\begin{EX}\label{ex:star-finite}
	Let $ S $ be an infinite graph-theoretical star. Then the edge tree set $ \tau(S) $ is a tree set which is chain-complete and splittable with finite $ C(s,s') $ for all $ s,s'\in\tau(S) $ but which contains a regular infinite splitting star.\\
\end{EX}

\begin{EX}\label{ex:branch-bounded}
Let $ B $ be the tree set with ground set $ \vm,\mv $ and $ \vs_n,\sv_n,\vt_n $ and $ \tv_n $ for every $ n\in\N $, with the following relations:
\begin{enumerate}
\item $ \vm\ge\vs_n,\vt_n $ and $ \mv\le\sv_n,\tv_n $ for all $ n\in\N $,
\item $ \vs_i\le\vs_j $ and $ \sv_i\ge\sv_j $ if and only if $ i\le j $,
\item $ \vt_i\le\tv_j $ if and only if $ i\ne j $,
\item $ \vs_i\le\tv_j $ and $ \sv_i\ge\vt_j $ if and only if $ i\le j $,
\item $ \vs_i\ge\vt_j $ and $ \sv_i\le\tv_j $ if and only if~$ i\ge j $ and $ i\ne j $.
\end{enumerate}
Then $ B $ is a tree set which is chain-complete and splittable with no infinite splitting star, but $ C(s_1,m) $ is infinite despite $ s_1 $ and $ m $ being regular.
\end{EX}

\section{Representing profinite tree sets}\label{sec:representations}

In~\cite{TreeSets} Diestel explored the various ways in which a tree set can be represented by a nested system of bipartitions of some groundset. Typically the groundset used to represent a tree set $ \tau $ is the set $ \vO(\tau) $ of consistent orientations of $ \tau $, or a suitable subset of $ \vO(\tau) $. The representation comes in the form of a map $ \ph\colon\tau\to\mathcal{B}(\vO(\tau)) $, where $ \mathcal{B}(X) $ denotes the set of (non-trivial) oriented bipartitions of $ X $, such that $ \ph $ is an isomorphism of tree sets between $ \tau $ and its image. As $ \mathcal{B}(X) $ is a regular separation system only regular tree sets can be represented by such bipartitions.

Diestel proved that the set of {\em directed} consistent orientations of a tree set can be used to represent that tree set, where an orientation $ O $ is {\em directed} if $ O $ is a directed set. A tree set is {\em ever-branching} if it contains no inclusion-maximal proper star of order $ 2 $.

\begin{THM}[\cite{TreeSets}]\label{thm:directed}
Let $ \tau $ be an ever-branching regular tree set, $ \tilde{\vO}=\tilde{\vO}(\tau) $ the set of all directed consistent orientations of $ \tau $, and $ \ph\colon\tau\to\mathcal{B}(\tilde{\vO}) $ the map
	\[ \ph(\vs):=\braces{\tilde{\vO}(\sv),\tilde{\vO}(\vs)}, \]
	where $ \tilde{\vO}(\vs):=\{O\in\tilde{\vO}\mid\vs\in O\} $. Then $ \ph $ is an isomorphism of tree sets between $ \tau $ and its image in $ \B(\tilde{\vO}) $.\qed
\end{THM}

In the remainder of this section we will show that every profinite tree set without a splitting two-star fulfills the assumptions of Theorem~\ref{thm:directed}, i.e. that it is ever-branching. Following that we shall use our insights from Section~\ref{sec:proof} to show that every regular profinite tree set can be represented by the bipartitions of its closed consistent orientations.

The first part of this is straightforward:

\begin{LEM}\label{lem:everbranching}
	A splittable tree set is ever-branching if and only if it has no splitting two-star.
\end{LEM}

\begin{proof}
	For the forward direction note that every splitting two-star of a tree set witnesses that that tree set is not ever-branching.
	
	For the backward direction, let $ \tau $ be a splittable tree set with no splitting two star and $ \{\vr,\vs\} $ a two-star in $ \tau $ with $ r\ne s $. As $ \tau $ is splittable there is a splitting star $ \sigma $ of $ \tau $ that contains separations $ \vrdash $ and $ \vsdash $ with $ \vr\le\vrdash\lne\svdash\le\sv $. By assumption~$ \sigma $ is not a two-star. Hence replacing $ \vrdash $ and $ \vsdash $ in $ \sigma $ with $ \vr $ and $ \vs $ yields a proper star which includes $ \{\vr,\vs\} $ as a proper subset, showing that $ \tau $ is ever-branching.
\end{proof}

As every profinite tree set is splittable by Theorem~\ref{thm:profinitechar} this shows that every regular profinite tree set without a splitting two-star can be represented as in Theorem~\ref{thm:directed}.

In fact for profinite tree sets the directed orientations have a very simple description:

\begin{LEM}\label{lem:directedsplitting}
	In a chain-complete tree set every directed consistent orientation has a greatest element.
\end{LEM}

\begin{proof}
	Let $ \tau $ be a chain-complete tree set, $ O $ a consistent directed orientation of~$ \tau $, and $ C $ a maximal chain in $ O $. The claim follows from the chain-completeness of $ \tau $ if $ O $ is the down-closure of $ C $. Suppose there is a $ \vs\in O $ with $ \vs\not\le\vt $ for every $ \vt\in C $. As $ C $ is maximal and $ O $ is consistent this means $ \vs\lne\tv $ for all $ \vt\in C $. Fix some $ \vt\in C $. As $ O $ is directed there exists $ \vr\in O $ with $ \vs,\vt\le\vr $. Then $ \vr\notin C $ by the choice of $ \vs $. In particular $ r\ne t $. Consider some $ \vu\in C $ with $ \vt\le\vu $. If $ \vr\le\vu $ then $ \vs $ would be trivial with witness $ r $; if $ \vr\le\uv $ then $ \vt $ would be trivial with witness $ r $; and finally $ \vr\ge\uv $ would contradict the consistency of $ O $. Therefore $ u $ and $ r $ must be related as $ \vr\ge\vu $. As this holds for every $ \vu\in C $ with $ \vt\le\vu $ we know that $ \vr $ is an upper bound for $ C $, contradicting the maximality of $ C $.
\end{proof}

Clearly the converse of Lemma~\ref{lem:directedsplitting} holds as well: every orientation of a tree set that has a greatest element is directed as witnessed by that element. Therefore we have established the following theorem, which is essentially a re-formulated special case of Theorem~\ref{thm:directed}:

\begin{THM}\label{thm:leaves}
	Let $ \tau $ be a regular profinite tree set with no splitting two-star, $ \vO'=\vO'(\tau) $ the set of all consistent orientations of $ \tau $ that have a greatest element, and $ \ph\colon\tau\to\mathcal{B}(\vO') $ the map
	\[ \ph(\vs):=\braces{\vO'(\sv),\vO'(\vs)}, \]
	where $ \vO'(\vs):=\menge{O\in\vO'\mid\vs\in O} $. Then $ \ph $ is an isomorphism of tree sets between $ \tau $ and its image in $ \B(\vO') $.\qed
\end{THM}

For regular profinite tree sets that do contain splitting two-stars we can still prove a succinct representation theorem. To do this we use the set $ \vOc(\tau) $ of all \textit{splitting} consistent orientations as a groundset rather than all directed orientations. For any splitting two-star the corresponding consistent orientation is then contained in $ \vOc(\tau) $ and can be used to distinguish the two elements of that star.

\begin{THM}\label{thm:representclosed}
	Let $ \tau $ be a regular profinite tree set, $ \vOc=\vOc(\tau) $ the set of all splitting consistent orientations of $ \tau $, and $ \ph\colon\tau\to\mathcal{B}(\vOc) $ the map
	\[ \ph(\vs):=\braces{\vOc(\sv),\vOc(\vs)}, \]
	where $ \vOc(\vs):=\menge{O\in\vOc\mid\vs\in O} $. Then $ \ph $ is an isomorphism of tree sets between $ \tau $ and its image in $ \B(\vOc) $.
\end{THM}

\begin{proof}
	We check the assumptions of Lemma~\ref{Isomorphism}.
	
	To see that $ \vOc(\vs) $ is non-emtpy for $ \vs\in\tau $, let $ C $ be a maximal chain in $ \tau $ containing $ \vs $. By Lemma~\ref{lem:extension} $ C $ extends to a consistent orientation~$ O $ of $ \tau $. As $ \tau $ is chain-complete by Theorem~\ref{thm:profinitecharregular} and $ C $ was chosen maximal $ C $ must have a greatest element $ \vm\in C $. This $ \vm $ is in fact even the greatest element of $ O $: for any $ \vr\in O $ either $ \vr\le\vm $ or $ \rv\ge\vm $ by consistency, but the latter case would contradict the maximality of $ C $. Therefore $ O $ is splitting and hence~$ O\in\vOc(\vs) $.
	
	Additionally $ \vOc=\vOc(\vs)\dot{\cup}\vOc(\sv) $ for each $ \vs\in\tau $, so $ \ph $ is indeed a map into~$ \mathcal{B}(\vOc) $.
	
	Clearly $ \ph $ commutes with the involution. It is also order-preserving: for $ \vr,\vs\in\tau $ with $ \vr\le\vs $ we have $ \vOc(\vs)\subseteq\vOc(\vr) $ by consistency and hence $ \ph(\vr)\le\ph(\vs) $. This shows that $ \ph $ is a homomorphism.
	
	It remains to prove that $ \ph $ is injective. For this consider $ \vr,\vs\in\tau $. If $ r=s $ then either $ \vr=\vs $, in which case there is nothing to show, or $ \vr=\sv $, in which case every orientation in $ \vOc(\vr) $ does not contain $ \vs $ and hence witnesses $ \ph(\vr)\ne\ph(\vs) $. Thus we may assume that $ r\ne s $. If $ \vr $ and $ \vs $ point away from each other then every orientation in $ \vOc(\vr) $ does not contain $ \vs $ by consistency and thus witnesses $ \ph(\vr)\ne\ph(\vs) $. If $ \vr $ and $ \vs $ point towards each other then every orientation in $ \vOc(\rv) $ contains $ \vs $ but not $ \vr $. Finally if $ \vr $ and $ \vs $ are comparable, say $ \vr\le\vs $, then by the splittability of $ \tau $ there is a splitting star $ \sigma $ of $ \tau $ such that $ \vr $ and $ \sv $ lie below different elements of $ \sigma $. The orientation of $ \tau $ induced by $ \sigma $ then contains $ \vr $ but not $ \vs $, witnessing that $ \ph(\vr)\ne\ph(\vs) $ and concluding the proof.
\end{proof}

We conclude with two remarks on Theorem~\ref{thm:representclosed}. First, by~\cite[Theorem~7.4]{ProfiniteASS} the splitting orientations of a regular tree set $ \tau $ are precisely those consistent orientations of $ \tau $ that are closed as a set in the inverse limit topology of $ \tau $; see~\cite[Section~7]{ProfiniteASS} for more. Thus in Theorem~\ref{thm:representclosed} we could equivalently have used the set of all closed consistent orientations as a groundset.

Finally, one can prove Theorem~\ref{thm:representclosed} with a slightly smaller groundset, namely without using splitting orientations that have three or more maximal elements: throughout the proof we have only used orientations with a greatest element, with the exception of the very last step. Given $ \vr,\vs\in\tau $ with $ \vr\lne\vs $ we find a splitting star $ \sigma $ with $ \vr $ and $ \sv $ below distinct elements of $ \sigma $. If $ \sigma $ is a two-star then its corresponding orientation of $ \tau $ contains $ \vr $ but not $ \vs $; and if $ \sigma $ has size three or greater it contains a separation $ \vt $ with $ \vr,\sv\le\vt $, in which case any orientation with a greatest element containing $ \tv $ also contains $ \vr $ and $ \sv $.

\bibliographystyle{plain}
%\bibliography{collective}

\begin{thebibliography}{1}

\bibitem{duality1inf}
N.~Bowler, R.~Diestel, and F.~Mazoit.
\newblock Tangle-tree duality in infinite graphs and matroids.
\newblock In preparation.

\bibitem{AbstractSepSys}
R.~Diestel.
\newblock Abstract separation systems.
\newblock {\em Order}, 35:157--170, 2018.

\bibitem{TreeSets}
R.~Diestel.
\newblock Tree sets.
\newblock {\em Order}, 35:171--192, 2018.

\bibitem{ProfilesNew}
R.~Diestel, F.~Hundertmark, and S.~Lemanczyk.
\newblock Profiles of separations: in graphs, matroids, and beyond.
\newblock {\em Combinatorica}, 39(1):37--75, 2019.

\bibitem{ProfiniteASS}
R.~Diestel and J.~Kneip.
\newblock Profinite separation systems.
\newblock {\em Order}, to appear.

\bibitem{TangleTreeAbstract}
R.~Diestel and S.~Oum.
\newblock Tangle-tree duality in abstract separation systems.
\newblock {S}ubmitted, arXiv:1701.02509, 2017.

\bibitem{GMX}
N.~Robertson and P.D. Seymour.
\newblock Graph minors. {X}. {O}bstructions to tree-decomposition.
\newblock {\em J.~Combin.\ Theory (Series B)}, 52:153--190, 1991.

\end{thebibliography}

\end{document}